\documentclass[12pt,a4paper]{preprint}
\usepackage{amsfonts,amsmath, charter,tikz}

\usepackage{ theorem}

\usetikzlibrary{matrix,arrows} 
\newtheorem{theorem}{Theorem}[section] 
 
\newtheorem{lemma}[theorem]{Lemma}  

{\theorembodyfont{\rmfamily} 
 
\newtheorem{example}{Example}[section]   }

\begin{document}

\def\paral{/\kern-0.55ex/}
\def\parals_#1{/\kern-0.55ex/_{\!#1}}
\def\bparals_#1{\breve{/\kern-0.55ex/_{\!#1}}}
\def\n#1{|\kern-0.24em|\kern-0.24em|#1|\kern-0.24em|\kern-0.24em|}
\newenvironment{proof}{
 \noindent\textbf{Proof}\ }{\hspace*{\fill}
  \begin{math}\Box\end{math}\medskip}

\newcommand{\A}{{\bf \mathcal A}}
\newcommand{\B}{{\bf \mathcal B}}
\def\C{\mathbb C}
\newcommand{\D}{{\rm I \! D}}
\newcommand{\dom}{{\mathcal D}om}
\newcommand{\pathR}{{\mathcal{\rm I\!R}}}
\newcommand{\Nabla}{{\bf \nabla}}
\newcommand{\E}{{\mathbb E}}
\newcommand{\Epsilon}{{\mathcal E}}
\newcommand{\F}{{\mathcal F}}
\newcommand{\G}{{\mathcal G}}
\def\g{{\mathfrak g}}
\newcommand{\HH}{{\mathcal H}}
\def\h{{\mathfrak h}}
\def\k{{\mathfrak k}}
\newcommand{\I}{{\mathcal I}}
\def\m{{\mathfrak m}}
\newcommand{\K}{{\mathcal K}}
\newcommand{\p}{{\mathbb P}}
\newcommand{\R}{{\mathbb R}}
\def\T{{\mathcal T}}
\newcommand{\pnabla}{{\nabla\!\!\!\!\!\!\nabla}}
\def\X{{\mathbb X}}
\def\Y{{\mathbb Y}}
\def\L{{\mathcal L}}
\def\1{{\mathbf 1}}
\def\half{{ \frac{1}{2} }}

\def\Ad{{\mathop {\rm Ad}}}
\def\ad{{\mathop {\rm ad}}}
\def\Conj{{\mathop {\rm Ad}}}
\newcommand{\const}{\rm {const.}}
\newcommand{\eg}{\textit{e.g. }}
\newcommand{\as}{\textit{a.s. }}
\newcommand{\ie}{\textit{i.e. }}
\def\s.t.{\mathop {\rm s.t.}}
\def\esssup{\mathop{\rm ess\; sup}}
\def\Ric{\mathop{\rm Ric}}
\def\div{\mathop{\rm div}}
\def\ker{\mathop{\rm ker}}
\def\Hess{\mathop{\rm Hess}}
\def\Image{\mathop{\rm Image}}
\def\Dom{\mathop{\rm Dom}}
\def\id{\mathop {\rm Id}}
\def\Image{\mathop{\rm Image}}
\def\Cyl{\mathop {\rm Cyl}}
\def\Conj{\mathop {\rm Conj}}
\def\Span{\mathop {\rm Span}}
\def\trace{\mathop{\rm trace}}
\def\ev{\mathop {\rm ev}}
\def\Ent{\mathop {\rm Ent}}
\def\tr{\mathop {\rm tr}}
\def\graph{\mathop {\rm graph}}
\def\loc{\mathop{\rm loc}}
\def\so{{\mathfrak {so}}}
\def\o{{\mathfrak {o}}}
\def\<{{\langle}}
\def\>{{\rangle}}
\def\span{{\mathop{\rm span}}}
\def\P{{\mathbb P}}
\def\su{{\mathfrak {su}}}

\renewcommand{\thefootnote}{}

\author{Xue-Mei Li}
\institute{Mathematics Institute, The University of Warwick, Coventry CV4 7AL, U.K.\\
Email address: xue-mei.li@warwick.ac.uk}
\title{Effective Diffusions  with Intertwined Structures}
\date{}
\maketitle

\begin{abstract}   
Let $p:N\to M$ be a surjective map of smooth manifolds. We are concerned with singular perturbation problems associated to a pair of second order positive definite differential operators with no zero order terms, that are intertwined by $p$. We discuss the associated random perturbations of stochastic differential equations and present a number of examples including  perturbation to geodesic flows and construction of a Brownian motion on $S^2$ through homogenisation of  SDE's on the Hopf fibration.  

  \end{abstract}
\bigskip

 {\it keywords.}  perturbation,   stochastic differential equations,  averaging,  homogenisation, geodesic flows,  Hopf fibration, frame bundle, holonomy bundle.\\
 

 \section{Introduction} 
The principal idea of singular perturbation is to deduce long term trends of a complex system   from that of a relatively simple one for which some observables are known.  We are mainly concerned with  random perturbation of stochastic differential equations with  intertwined diffusion structures. 
A  {\it diffusion operator} $\B$ on a smooth finite dimensional manifold $N$ is a second order differential operator  with positive definite symbol and vanishing zero order term. Its symbol $\sigma^\B$ is a real valued bilinear map on the cotangent bundle, determined by
 $$\sigma^\B(dg_1,dg_2)={1\over 2}\left(\B(g_1g_2)-(\B g_1)g_2-g_1(\B g_2)\right),$$
 where $g_1,g_2$ are  real-valued $C^2$ functions on $N$.
Let $n=\dim(N)$. In a local coordinate system, let
 $$\B g={1\over 2}\sum_{i,j=1}^n a_{i,j}{\partial^2 g\over \partial y_i\partial y_j}+\sum_{k=1}^n b_k{\partial g\over \partial y_k}$$ where  $a_{i,j}$  and $b_k$ are smooth functions with the $n\times n$ matrix valued function $(a_{i,j})$ positive symmetric. Then for $y\in N$, 
   $$\sigma^\B_y(dg_1(y), dg_2(y))={1\over 2}\sum_{i,j=1}^na_{i,j} (y){\partial g_1\over \partial y_i}(y){\partial g_2\over \partial y_j}(y).$$
  The operator $\B$  is said to be elliptic if the symbol is strictly positive and to have constant rank if the rank of its symbol $\sigma^\B_u$ is  constant in $u$.
   
 Let $p:N\to M$ be a smooth onto map between smooth manifolds $N$ and $M$.  Denote by $Tp: TN\to TM$ the differential of $p$. 
 If $\B$ is an operator on $N$ and $\A$ an operator on $M$ such that  $$\B(f\circ p)=(\A f)\circ p$$ for all real valued $C^2$ functions $f$ on $M$,
we say that $\B$ and $\A$ are {\it intertwined} or $\B$ is over $\A$. The simplest intertwining is given by projection of a product space to one of the factor spaces. 


Let $T_up:T_uN\to T_{p(u)}M$ be the differential of the map $p$ at $u$, a linear map between tangent spaces.  Its kernel is said to be a vertical  tangent space, a subspace of
the tangent space  $T_uN$  and is denoted by $VT_uN$. The vector bundle $VTN=\cup_u VT_nN$ is called the vertical tangent bundle.
   We say that a diffusion operator $\B^0$ on $N$ is {\it vertical} if $\B^0(f\circ p)=0$ for any $C^2$ function $f$ on $M$.
   
Suppose that $\B$ is over $\A$ and that $\{\sigma_x^\A, x\in M\}$ has constant rank then there is a unique smooth lifting map 
$$\h_u: \Image[\sigma_{p(u)}^\A]\subset T_{p(u)}M\to \Image(\sigma^\B_u)\subset T_uN$$ such that $T_up\circ \h_u$ is the identity map (Proposition 2.1.2 in Elworthy-LeJan-Li\cite{Elworthy-LeJan-Li-book-2}). The image $\h_u$ induces a smooth distribution, called the horizontal distribution associated to $\A$. 
 In the case of $\A$  elliptic let $HT_uN$ be the image of $\h_u$ then
$$T_uN=HT_uN\oplus VT_uN.$$
In the case where $p:N\to M$ is a principal bundle this induces an Ehresmann connection on $N$ with a corresponding connection 1-form.

Let $\A=\half\sum_{i=1}^m L_{X_i}L_{X_i}+L_{X_0}$,  where $X_i$ are vector fields  and $L_{X_i}$  denotes Lie differentiation in the direction of $X_i$.  For short we also write $\A=\half\sum_{i=1}^m X_i^2+X_0$.
Suppose that  $X_0\in \Image\sigma^\A$ we say $\A$ is cohesive. An elliptic operator is cohesive. The horizontal lift of $\A$ is:   
$$\A^H=\half \sum {\tilde X_i}^2+{\tilde X_0}$$ where $\tilde X_i(u)$ is the horizontal lift of the tangent vector $X_i(p(u))$.
By Theorem 2.2.5 in \cite{Elworthy-LeJan-Li-book-2},  there is a unique vertical diffusion $\B^v$ such that $\B=\A^H+\B^v$. If $u$ is a regular point of the map $p$,  
the $\B^v$ diffusion starting from $u$ stays in the sub-manifold $p^{-1}(p(u))$. 
The vertical diffusion operator is `elliptic' if
 $\sigma_u^{\B^v}: VT_uN\times VT_uN\to \R$ is strictly positive definite.

Let $\A$ be a cohesive diffusion on $M$ and $\B^0$ a vertical diffusion.  Let $\L^\epsilon={1\over \epsilon}\A^H+\B^0$, or $\L^\epsilon=\A^H+{1\over \epsilon}B^0$ where $\epsilon$ is a real number which we take to zero. We would like to understand the asymptotic behaviour of the solutions of
${\partial \over \partial t}=\L^\epsilon$
as $\epsilon $ goes to zero.

For simplicity write $\L^\epsilon={1\over \epsilon}\L_0+\L_1$. Let us expand $f^\epsilon(t,y)$, a solution of the parabolic differential equation
 ${\partial \over \partial t}f^\epsilon(t,y)=\L^\epsilon f^\epsilon (t,y)$ in $\epsilon$: $f^\epsilon={1\over \epsilon} f_0+f_1+ \epsilon f_2+o(\epsilon)$.
What can we say about $f_i$? Is there a diffusion operator $\bar L$ such that  $f_1$ solves ${\partial \over \partial t}=\bar L $? See e.g. the book of  Arnold \cite{Arnold89} in the context of perturbation to Hamiltonian systems and a recent book of Stuart-Paviliotis \cite{Pavliotis-Stuart08} on multi-scale methods.
 
We now introduce notations concerning  Markov processes associated to diffusion operators. Let $(\Omega, \F, \F_t, P)$ be a filtered probability space. 
Let $\phi_t(y, \omega)$ be a family of  strong Markov stochastic process with values in a manifold $N$  with $\phi_0(y)=y$.  The $\omega$ variable will be suppressed for simplicity. Define a linear operator $P_t$ on the space of bounded functions by $P_t f(y)=\E f(\phi_t(y, \omega))$, where $\E$ denotes integration with respect to $P$. 
Then $(P_t, t\ge 0)$ is a semigroup of operators with $P_0$ the identity map. Its infinitesimal generator $\L$ is defined by $\L f=\lim_{t\to 0}{P_t f-f\over t}$
with domain the set of functions such that the limit exists. For such $f$, $P_t f$  solves the  parabolic equation ${\partial \over \partial t}=\L$ with initial function $f$.  
On the other hand given any diffusion operator $\L$, there is a strong Markov process whose infinitesimal generator is $\L$. This can be seen by introducing a H\"ormander form representation of $\L$ and a stochastic differential equation. When $\L$ is reasonably smooth the Markov process is continuous in $t$ for almost surely all $\omega\in \Omega$ and the Markov process is said to be a diffusion (process) with generator $\L$.  
If $\L={1\over 2}\Delta$, where $\Delta$ is the Laplacian operator for a Riimannian metric, we say the diffusion process is a Brownian motion for that metric.  

The  dynamic picture of the perturbation problem is as following. Let $y_t^\epsilon$ be a diffusion operator with initial value $y_0$ associated to $\L_0+\epsilon \L_1$. The  $y_t^\epsilon$ process follows roughly the orbit determined by $\L_0$ with negligible differences. However on a large time scale of order $1/\epsilon$,  the deviation of the perturbed orbit from the unperturbed one becomes visible.  In general we ask whether there is a function of $y_t^\epsilon$ that does not change with time when $\epsilon=0$.
If so, denote this function by $F$. The variable $F(y_t^\epsilon)$ should vary slowly when $\epsilon \to 0$ and in the limit  $F(y_{t\over \epsilon}^\epsilon)$  may converge
to a Markov process whose probability distribution is determined by a diffusion operator $\bar L$. It is desirable to find all conserved quantities and their explicit probability distribution, i.e. the limiting diffusion operator $\bar L$, which is said to be the effective motion. There  is extensive literature on this and we refer to the following books and the references therein:  Bensoussan-Lions-Papanicolau \cite{Bensoussan-Lions-Papanicolaou78} and Freidlin-Wentzell \cite{ Freidlin-Wentzell98}.

We now consider the problem at the level of stochastic differential equations. Every diffusion operator, if sufficiently smooth, can be represented as sum of squares of vector fields,
the so called H\"ormander form representation. We write the two diffusion operators  $\L_0$ and $\L_1$ in  H\"ormander form: 
$\L_0={1\over 2}\sum_{i=1}^m {X_i}^2+{X_0}$, and  $\L_1={1\over 2}\sum_{i=1}^m {Y_i}^2+{Y_0}$. 
Let $(b_t^i, w_t^j)$ be independent one dimensional Brownian motions on a given filtered probability space $(\Omega, \F, \F_t,P)$ with the usual assumptions.
Let $\phi_t^\epsilon(y)$ denote the solution  to the stochastic differential equations (SDE) driven by the vector fields ${1\over \sqrt \epsilon}X_i, {1\over \epsilon} X_0,   Y_i,  Y_0$, $i=1,\dots, m$, with initial point $y$:
$$dy_t^\epsilon={1\over \sqrt  \epsilon}  \sum_i X_i(y_t^\epsilon) \circ db_t^i+{1\over \epsilon} X_0(y_t^\epsilon)dt+
\sum_j Y_j(y_t^\epsilon)\circ dw_t^j+Y_0(y_t^\epsilon)dt.$$
The solutions are continuous Markov processes (diffusion processes) with  generator  $\L^\epsilon$.

 {\bf Main Results.}
  The structure of the paper is as following.  In section 2, we present several  intertwining examples to illustrate the terminologies. 
    In section  \ref{Hopf-section} SDE's on the Hopf vibration are investigated.  We construct stochastic process with generator a hypoelliptic horizontal diffusion on $S^3$ from two vector fields one of which the vertical vector field induced by the circle action and the other a horizontal vector field induced from an element of the Lie algebra of norm $1$. In particular we construct a Brownian motion on $S^2$ through homogenisation. 
 
 
 In section \ref{OM-section} the state space is the frame bundle of a Riemannian manifold.  The frame bundle is closely related to the tangent bundle of the manifold, notably suitable first order differential equations on the frame bundle correspond to second order differential equations on the manifold.  It is also the natural space to record the position and orientation of a particle. For intertwined diffusion models on  the  frame bundle of a complete Riemannian manifold $M$ there are two basic slow motions:  the horizontal diffusions and the vertical diffusions. The horizontal diffusions appear naturally in Malliavin calculus.   The solution flow of a vertical diffusion can be considered  as a random evolution of a linear frame and is an interesting object in geometry, e.g.  it was examined in Brendle-Schoen \cite{Brendle-Schoen} in connection with the question as to whether positive isotropic curvature condition  is preserved by R. Hamilton's ODE.

In section \ref{section-vertical} Perturbations to vertical diffusions on the orthonormal frame bundle, whose projection to the manifold is a fixed point are studied. The effective motion on the manifold is a diffusion which is of the same `type' as the perturbation (Theorem \ref{level-thm}). Special care has to be taken in this case to avoid assumptions on the injectivity radius of the orthonormal frame bundle. 
 In section \ref{OU-section} a suitable perturbation model to the geodesic flow, a first-order ODE on the frame bundle, is developed. The perturbation is of Ornstein-Uhlenbeck type and in the limit we see a diffusion with generator ${4\over n(n-1)}\Delta_H$,  $\Delta_H$ being the horizontal Laplacian, and a rescaled Brownian motion to the manifold. This study relates to approximation of Brownian motions and by extension that of stochastic differential equations. We would also like to compare this to the philosophy in Bismut \cite{Bismut-hypoelliptic-Laplacian}, that
$\ddot x={1\over T}(-\dot x+\dot w)$ interpolates between classical Brownian motion $(T\to 0)$ and the geodesic flow ($T\to \infty$).
 In  section \ref{section-holonomy}  perturbation to the the semi-elliptic horizontal flow is considered and we obtain an effective motion that is transversal to the holonomy bundle.   
 
   
  In terms of ellipticity our SDEs have the following features.  The unperturbed system can be elliptic or hypoelliptic. The perturbations could be  elliptic, hypoelliptic or degenerate.  We used two types of scalings:  the standard scaling and the scaling of `Ornstein-Uhlenbeck type'.  A related problem on commutation of linearisation with averaging is discussed in a paper in preparation \cite{Li-OM-2}.  

\section{Some Basic Examples}
Intertwined structures occur naturally.  One such standard example is  projection of  a product space. The other example is a principal bundle  $p:N\to M$ with a group action $G$ which acts freely. The latter introduces a twist in the product structure. These will include projections of groups to their quotient groups and that of manifolds to their moduli spaces. See \cite{Elworthy-LeJan-Li-book-2} and Liao \cite{Liao-00} for a discussion of diffusion processes  on symmetric spaces. We give some examples which illustrate the procedure of averaging with intertwined structures and the local structure of frame bundles.
Let us begin with the trivial example of a cylinder $\R\times S^1$. Let $z$ denote the $S^1$ direction, $p(x,z)=z$, $\A={\partial^2 \over \partial z^2}$,
and $\B=\sin z {\partial\over \partial x}+{\partial ^2\over \partial z^2}$. Then $\A^H={\partial^2 \over \partial z^2}$.
 The effective motion associated to associated to hypoelliptic flows $\sin z {\partial\over \partial x}+{1\over \epsilon}{\partial ^2\over \partial z^2}$ converges to a simple completely degenerate motion.

 {\bf Example 1.}
Take $N=\R^3$ with the Heisenberg group structure. For $(x,y,z)\in \R^3$ define $p(x,y,z)=(x,y)$. 
Let   $Y_1={\partial \over \partial x}+{1\over 2} y {\partial \over \partial z} $ and $Y_2={\partial \over \partial y} -{1\over 2} x{\partial \over \partial z} $ be the left invariant vector fields on  $N$ associated to $(1,0,0)$ and $(0,1,0)$. 
 Let $\A^H= \half (Y_1^2+Y_2^2)$.
Consider $\B^0= \half{\partial ^2\over \partial z^2}-{\partial \over \partial z}$. Then
 $\L^\epsilon= \half (Y_1^2+Y_2^2)+{1\over \epsilon}\B^0$ is over  the cohesive operator $\A={1\over 2}( {\partial^2 \over \partial x^2}+{\partial^2 \over \partial y^2})$ on $\R^2$. The associated horizontal lift  is $(u,v)\mapsto (u,v,  \frac{1}{2} xv-\frac{1}{2}yu)$. 
The vector fields $Y_1, Y_2$ are respectively  the horizontal lifts of ${\partial \over \partial x} $ and
${\partial \over \partial y}$ and $\A^H$ is the horizontal lift of $\A$ which is seen to have constant rank $2$. The `slow part' of the fully elliptic $\L^\epsilon$ diffusion  converges weakly to the $\A^H$ diffusion, a hypo-elliptic diffusion, as expected. 
We illustrate what does it mean by the `slow part' by the following trivial, but explicit, SDE:   $${\begin{split} dx_t^\epsilon=&\cos(z_t^\epsilon)\circ db_t^1-\sin(z_t^\epsilon)\circ db_t^2, dy_t^\epsilon=\sin(z_t^\epsilon)\circ db_t^1+\cos(z_t^\epsilon)\circ db_t^2\\
dz_t^\epsilon =&{1\over 2} (x_t^\epsilon\sin(z_t^\epsilon)-  y_t^\epsilon \cos(z_t^\epsilon))\circ db_t^1+{1\over 2} (x_t^\epsilon \cos(z_t^\epsilon)+ y_t^\epsilon \sin(z_t^\epsilon))\circ db_t^2\\
& +{1\over \sqrt \epsilon}  dw_t-{1\over \epsilon}  z_t^\epsilon dt.  \end{split}}$$
Two obvious slow variables are $(x_t^\epsilon, y_t^\epsilon)$. In our terminology the slow part is the solution to the following SDE parametrized by $z_t^\epsilon$:
$${\begin{split} d \tilde x_t^\epsilon=&\cos(z_t^\epsilon)\circ db_t^1-\sin(z_t^\epsilon)\circ db_t^2, d\tilde y_t^\epsilon=\sin(z_t^\epsilon)\circ db_t^1+\cos(z_t^\epsilon)\circ db_t^2\\
d\tilde z_t^\epsilon =&{1\over 2} (\tilde x_t^\epsilon\sin(z_t^\epsilon)-  \tilde y_t^\epsilon \cos(z_t^\epsilon))\circ db_t^1+{1\over 2} (\tilde x_t^\epsilon \cos(z_t^\epsilon)+ \tilde y_t^\epsilon \sin(z_t^\epsilon))\circ db_t^2.  \end{split}}$$
The law of the process in independent of $\epsilon$ and is ${1\over 2}({\partial^2\over \partial x^2}+{\partial^2\over \partial y^2})-y 
{\partial^2\over \partial x\partial z}+x{\partial^2\over \partial x\partial y}+{1\over 8}(x^2+y^2){\partial^2\over \partial z^2} $, c.f. Example \ref{example-right}
for an example on the frame bundle that is in the same spirit.
 The third component of the Markov process associated to $\L^\epsilon$ converges to the stochastic area of the limits of the first two components: $\half \int_0^{t}\tilde x_s^\epsilon d\tilde y_s^\epsilon-\half \int_0^{t }\tilde y_s^\epsilon d\tilde x_s^\epsilon$  converges. This means taking stochastic area and taking $\epsilon\to 0$ commute, as expected.
 
This example can extend to the case of a general connection given by 
$\h_{(x,y,z)}(u,v)=(u,v,r_1u+r_2v)$ allowing the functions $r_i$ to depend on $z$. Assuming that  $({\partial r_1 \over \partial y}-{\partial r_2 \over \partial x})^2$ is strictly positive, $\h$ is determined by the operator
$\A^H:=\half  ( {\partial \over \partial x}+r_1{\partial \over \partial z} )^2+\half ( {\partial \over \partial y}+r_2{\partial \over \partial z} )^2$ and $\A$, page 21 Elworthy-LeJan-Li\cite{Elworthy-LeJan-Li-book-2}. Let $X_1={\partial \over \partial x}+r_1{\partial \over \partial z}$ and $X_2={\partial \over \partial y}+r_2{\partial \over \partial z}$.
The first prolongation of $\span\{X_1, X_2\}$, \ie  $\span\{X_1,X_2, [X_1,X_2]\}$ has full rank at each point. 
Consider a function $\alpha$ such that the invariant measure $\mu$ of  $(\gamma-r_1^2-r_2^2){\partial^2\over \partial z}+\alpha {\partial \over \partial z}$ is the standard Gaussian measure. Let 
$\gamma$ be such that $\gamma-r_1^2-r_2^2>c>0$ some $c$.
 The $\A^H+{1\over \epsilon^2}(\gamma-r_1^2-r_2^2){\partial^2\over \partial z}+{1\over \epsilon}\alpha {\partial \over \partial z}$ diffusion converges to an elliptic diffusion
on the Heisenberg group if $r_i$ are not constants a.s. with respect to $\mu$.

\medskip

{\bf Example 3.}    A non-relativistic quantum 
 mechanical diffusion lives naturally in $\R^3\times SO(3)$, the orthonormal frame bundle of $\R^3$. Its spatial projection lives in $\R^3$.  
Studies associated to quantum mechanical equations, mainly the continuity equation describing the probability density of the quantum equation,  have  intertwined structures on $p: \R^3\times SO(3)\to \R^3$. I am grateful to D. Elworthy to bring my attention to the paper
of Wallstrom \cite{Wallstrom} where  limits of stochastic processes in
$\R^3\times SO(3)$ are discussed.
The Bopp-Haag-equations  have one free parameter $I$ and its solutions converge to that of an equation with Pauli Hamiltonian as $I\to 0$. The Bopp-Haag-Dankel stochastic mechanical diffusions $\R^3\times SO(3)$ were introduced by Dankel, describing a diffusion particle with 
definite position and orientation.
The Bopp-Haag -Dankel diffusions on $\R^3\times SO(3)$ are given by a simple SDE with  drift  given by a Pauli spinor (solution of quantum equation associated with Pauli Hamiltonian with parameter $I$). In \cite{Wallstrom} it was shown  that for spin ${1\over 2}$ wave functions and regular potentials the process parametrized by $I$ converge to a Markovian process onto $\R^3$, due to the averaging out of the orientational motion. The spatial projection describes  the spatial motion of the particle without its orientation. 

 

{\bf Example 4.}  Let $G=SO(n)$ and $\pi:\R^n\times G\to \R^n$  the projection to its first component.  Let $\g=\so(n)$ be the Lie algebra of the Lie group $G$.
For each $x\in \R^n$, let
$h_x: T_x\R^n\sim \R^n\to \g$ be a linear map varying smoothly in $x$. The map $(x,v)\mapsto (x, h_x(v))$ can be considered as the horizontal lifting map through $(x,I)$ where $I$ is the identity matrix. 
This induces on $\R^n$ a non-trivial covariant differentiation $\nabla$. 
Let $e\in \R^n$, 
consider the SDE
\[{
\begin{split}
dx_t &=\epsilon_1 g_t  \circ d b_t+\epsilon g_t e\,dt\\
dg_t &= \epsilon_1 h_{x_t^\epsilon} (g_t \circ  d b_t) g_t+\epsilon h_{x_t} (g_t e) g_t dt
+\sqrt \delta \sum_{k=1}^p Z_k(x_t, g_t) \circ d w_t^k+\delta Z_0(x_t, g_t)dt.
\end{split}}
\]
where $( b_t^i,  w_t^k)$  are independent 1-dimensional Brownian motions, $ b_t=( b_t^1, \dots,  b_t^n)$, $ w_t=( w_t^1, \dots,  w_t^p)$, and $Z_k:\R^n\times G\to TG$ with $Z_k(x,g)\in T_gG$. We denote by $\circ $ Stratonovich integration, which must be used in the manifold setting and the correction term should be computed.  When $h=0$ this corresponds to the flat connection. We consider three types of scalings: 
1)  $\delta=1$, $\epsilon_1=\sqrt  \epsilon$ and $\epsilon\to 0$; 
2) $\epsilon_1=\epsilon=1$ and $\delta\to 0$. For the third type take $\epsilon_1=0$, $\epsilon=1$ and $\delta\to \infty$. 
In case 1) it turns out that the solution $x_t$ is a slow variable, despite the involvement of $g_t$. 

%
 
 Let us now describe the model using the language of orthonormal frame bundles. Let 
$v\mapsto \h_{(x,g)}(v)$ be the horizontal lifting map. Then $\h_{(x,g)}(v)=\h_{(x,I)}(v)g$ where $I$ is the identity matrix.
Denote by $h_x(v)$  the lifting map at $(x, I$).  Define $\theta_{(x,g)}(v, w)=g^{-1}(v)$. 
For $(v, A)\in \R^n\times \g$, define $\varpi_{(x,I)}(v,A)=A-h_x(v)$ and
$$\varpi_{(x,g)}(v, Ag)=g^{-1} \varpi_{(x,I)}(v,w)g=g^{-1}Ag-g^{-1}h_x(v)g.$$
The above example is the orthonormal frame bundle,  $\R^n\times SO(n)\to \R^n$, of $\R^n$ with the group $G=SO(n)$ acting on the right of $\R^n\times SO(n)$.  This  describes the local structure of the orthonormal frame bundle of a Riemannian manifold. The frame bundle of
$\R^n$ can also be represented as the Euclidean group. See section \ref{OM-section} for limiting theorems concerning the mentioned SDE in the context of a general manifold.


\section{Homogenisation on the Hopf Fibration}
\label{Hopf-section}
Hopf fibration occurs in multiple situation in physics: in quantum systems and in mechanics, see \eg Urbantke \cite{Urbantke} for an account. 
Hopf fibration is  the principle bundle $\pi: S^3\to S^2$ with $S^1$ acting on the right. Here $S^n$ denotes the $n$-sphere.
As pointed out by M. Berger in 1962, this is a non-trivial collapsing manifold. The sphere  $S^3$ is equipped with a metric inherited from $\R^4$.   The collapsing was achieved by shrinking the length by a scale of $\epsilon$ along the Hopf fibration direction and leaving the orthogonal directions unchanged. In a paper in preparation we study the dynamics associated to collapsing  manifold \cite{Li-Hopf}.

 It is convenient to consider the representation by unitary groups: $S^3$ is identified  with $SU(2)$, $S^1$ with $U(1)$ and $S^2$ with $SU(2)/U(1)$.
A typical element of $SU(2)$ may be expressed as  $(z,w)$, where $z, w\in \C$ are such that $|z|^2+|w|^2=1$, 
or as a matrix  $\left( \begin{array} {cr} z  &-\bar w\\  w  &\bar z
\end{array} \right)$. The right action by $e^{i\theta}\in U(1)$ is  $(z,w)\mapsto (e^{i\theta} z, e^{i\theta}w)$, which can be considered as 
  right multiplication in the group $SU(2)$ by elements of the form $e^{i\theta}\sim (e^{i\theta}, 0)$:
$$\left( \begin{array} {cr} z  &-\bar w\\  w  &\bar z
\end{array} \right)\mapsto \left( \begin{array} {cr} z  &-\bar w\\  w  &\bar z
\end{array} \right)\left( \begin{array} {cc} e^{i\theta}   &0\\ 0  & e^{-i\theta}
\end{array} \right).$$
The Hopf map $\pi: SU(2)\to S^2$ is a submersion,
$$\pi(z,w)=(Re(2z\bar w), Im(2z\bar w), |z|^2-|w|^2).$$
%
The map $T_u \pi$  can be better visualised if $S^3$ is treated as a subset of $\R^4$, writing $z=y_1+iy_2, w=y_3+iy_4$,  $y=(y_1,y_2, y_3, y_4) \in R^4$, 
$$T_y\pi=2\left(\begin{array}{cccc} 
y_3&y_4, &y_1&y_2\\
-y_4&y_3&y_2,&-y_1\\
y_1,& y_2, &-y_3, &-y_4\end{array}\right).$$
The vertical tangent spaces are the kernels of $T\pi$. It is easy to check that he vector field
$V(y_1,y_2, y_3, y_4)=-y_2\partial_1+ y_1\partial_2 -y_4 \partial_3 +y_3\partial_4 $ is vertical.
Back to the principal bundle picture, $V((z,w)):=(iz,iw)$ is the fundamental vertical vector field, associated to $i$ in  the Lie algebra of $U(1)$.  

The Lie algebra $\su(2)$ is the set of 
matrices such that $A+\bar A^T=0$ and with zero trace:
$$\left( \begin{array} {cc} ia  &\beta\\ -\bar \beta   &-ia
\end{array} \right), \qquad a\in \R, \beta\in \C. $$
 We take a real valued inner product on $\su(2)$,  $\<A,B\>:={1\over 2}\trace AB^*$, and  the following orthonormal basis: 
$$X_1=\left( \begin{array} {cc}i  &0\\ 0 &-i
\end{array} \right), \quad X_2=\left( \begin{array} {cc} 0  &-1\\  1&0
\end{array} \right), \qquad X_3=\left( \begin{array} {cc} 0  &i\\i &0
\end{array} \right).$$

Note that $X_1$ is adjoint invariant under the circle action and so is the linear span of $\{X_2, X_3\}$.
Denote by $X_i^*$ the left invariant vector fields associated to $X_i$. Let us define
  a distribution $D=\span\{X_2^*, X_3^*\}$, which is obviously left invariant with respect to the group action on $S^3$. The span of the left invariant vector fields is also right invariant under the circle action. This is due to the fact that $u e^{i\theta}X_i\in D_{u e^{i\theta}}=u(X_ie^{-2i\theta})e^{i\theta} \in D_u e^{i\theta}$ for $i=2,3$. Then $T_uS^3=[\ker T_u\pi]\oplus D_u$ defines  an Ehresmann connection on the principal bundle and a horizontal lifting map.

 Let $\nabla^L$ be the left invariant linear connection and $\nabla$ the Levi-Civita connection for the bi-invariant Riemannian metric on the Lie group $SU(2)$.  Denote by $\Delta$ the Laplacian on $S^3$. Let $\Delta_H=\sum_{i=2}^3\nabla ^Ldf(X_i^*,X_i^*)=\sum_{i=2}^3 \L_{X_i}L_{X_i}$ be the hypoelliptic Laplacian corresponding to the Horizontal distribution generated by the left invariant vector fields $\{X_2^*, X_3^*\}$.

\subsection{Construction of Brownian Motion on $S^2$ by Homogenisation}
 
  Let  $Y_0$ to be  vector in $\span\{X_2, X_3\}$.  
   Since $Y_0e^{i\theta}$ remains in $\span\{X_2, X_3\}$  the SDE  $du_t^\epsilon = u_t^\epsilon Y_0 g_t^\epsilon dt+{1\over \sqrt \epsilon} u_t^\epsilon X_1\circ db_t$  on $S^3$ makes sense
   for any stochastic process $g_t^\epsilon \in U(1)$.  Note that $\{X_1, X_2, X_3\}$ is a Milnor frame \cite{Milnor} with structural constants $(-2,-2,-2)$,
$$[X_1,X_2]=-2X_3, \quad [X_2, X_3]=-2 X_1, \quad [X_3,X_1]=-2X_2.$$
 If $Y_0=c_2X_2+c_3X_3\not =0$, $\span\{X_1, Y_0, [Y_0, X_1]\}=\span\{X_1, X_2, X_3\}$. By the structural equations $[Y_0, X_1]=2c_2X_3-2c_3X_2$ and $\{X_1, Y_0, [Y_0, X_1]\}$
is linearly independent following from the non-degeneracy of the matrix
   $$\left(
\begin{array}{ccc}
1&0&0\\
0  &c_2 &-2c_3\\ 0&c_3&2c_2
\end{array}	\right).$$
It follows that the SDE under discussion is hypoelliptic.

The Hopf map $\pi$ projects a  curve $u_t$ in $SU(2)$ to one in $S^2$. A curve $x_t$ in $S^2$ lifts to a horizontal curve
$\tilde x_t$ in $SU(2)$ through the horizontal lifting map induced by the Ehresmann connection.
  \begin{theorem}
\label{Hopf-theorem}
Let $(b_t)$ be a one dimensional Brownian motion  and take $u_0\in SU(2)$.  Let  $(u_t^\epsilon, g_t^\epsilon)$ be the solution to the following SDE on $SU(2)\times U(1)$,  with $u_0^\epsilon=u_0$
and $g_0^\epsilon=1$,
$$du_t^\epsilon = u_t^\epsilon Y_0 g_t^\epsilon dt+{1\over \sqrt \epsilon} u_t^\epsilon X_1\circ db_t, \qquad dg_t^\epsilon={1\over \sqrt \epsilon} g_t^\epsilon X_1 \circ db_t.$$
Let $x_t^\epsilon=\pi(u_t^\epsilon)$ and $\tilde x_t^\epsilon$ its horizontal lift.  Then  $\tilde  x_{t\over \epsilon}^\epsilon$ converges in probability to the hypoelliptic diffusion with generator  $\bar \L F={1\over 2}|Y_0|^2 \Delta_H$.
If $Y_0$ is a unit vector,  $x_{t\over \epsilon}^\epsilon$ converges in law to the Brownian motion on $S^2$.

\end{theorem}
\begin{proof}
Let $a_t^\epsilon \in S^1$ be such that $u_t^\epsilon=\tilde x_t^\epsilon a_t^\epsilon$ where $\tilde x_t^\epsilon $ is the horizontal lift of $x_t^\epsilon$ through $u_0$
using the connection determined by $\{X_2^*, X_3^*\}$.  Then $a_0^\epsilon=1$ and 
$$d \tilde x_t^\epsilon= TR_{(a_t^\epsilon)^{-1}} \circ  du_t^\epsilon +( (a_t^\epsilon)^{-1} \circ da_t^\epsilon)^*(u_t^\epsilon). $$
All stochastic integration involved in the above equation are Stratonovich integrals. Thus
$$d \tilde x_t^\epsilon= TR_{(a_t^\epsilon)^{-1}} \left( u_t^\epsilon Y_0 g_t^\epsilon dt+{1\over \sqrt \epsilon} u_t^\epsilon X_1\circ db_t\right) +
( (a_t^\epsilon) \circ d(a_t^\epsilon)^{-1})^*(\tilde x_t^\epsilon). $$
Since $\tilde x_t^\epsilon$ is horizontal, $\omega(d\tilde x_t^\epsilon)=0$, we obtain
$$(a_t^\epsilon) \circ d(a_t^\epsilon)^{-1}=-\varpi_{\tilde x_t^\epsilon}\left({1\over\sqrt \epsilon} u_t^\epsilon X_1 (a_t^\epsilon)^{-1}\circ db_t\right)=-{1\over\sqrt \epsilon} a_t^\epsilon X_1 (a_t^\epsilon)^{-1}\circ db_t .$$
It follows that $d(a_t^\epsilon)^{-1}  =-{1\over\sqrt \epsilon} a_t^\epsilon X_1 (a_t^\epsilon)^{-1}\circ db_t$, $(a_t^\epsilon)=(g_t^\epsilon)$, and
$$d \tilde x_t^\epsilon= u_t^\epsilon Y_0  dt+{1\over \sqrt \epsilon} u_t^\epsilon X_1(g_t^\epsilon)^{-1} \circ db_t
-{1\over \sqrt\epsilon} \tilde x_t^\epsilon a_t^\epsilon X_1 (a_t^\epsilon)^{-1}\circ db_t=\tilde x_t^\epsilon g_t^\epsilon Y_0dt. $$
Since there is no Stratonovich correction term for $dg_t= g_t  X_1  \circ db_t$, the corresponding infinitesimal  generator is ${1\over 2}\Delta_{S^1}$ where $ \Delta_{S^1}$ 
is the Laplacian on $S^1$. Let $F: S^3\to \R$ be any smooth function. Since $Y_0\in \span\{X_2, X_3\}$,
$${\begin{split}F(\tilde x_t^\epsilon)&=F(u_0)+\int_0^t dF(\tilde x_s^\epsilon g_s^\epsilon Y_0)ds
=F(u_0)+\sum_{j=2}^3\int_0^t dF( \tilde x_s^\epsilon X_j)\<\tilde x_s^\epsilon X_j\,\tilde x_s^\epsilon g_s^\epsilon Y_0\>ds\\
&=F(u_0)+\sum_{j=2}^3\int_0^t dF( \tilde x_s^\epsilon X_j)\<X_j, g_s^\epsilon Y_0\>ds.
\end{split}}$$
The two real valued functions on $S^1$, $g\mapsto \<X_2, gY_0\>$ and  $g\mapsto \<X_3,gY_0\>$, are eigenfunctions of $\Delta_{S^1}$.  Then for  $j=2,3$,
$${\begin{split} &dF(\tilde x_t^\epsilon X_j )\<X_j, g_t^\epsilon Y_0\> -dF(u_0X_j)\<X_j,Y_0\>
\\&= \sum_{k=2}^3\int_0^t \nabla^L dF(\tilde x_s^\epsilon X_k, \tilde x_s^\epsilon X_j)
\<X_j, g_s^\epsilon Y_0\>\<X_k, g_s^\epsilon Y_0\>ds 
+{1\over\sqrt \epsilon}\int_0^t dF(\tilde x_s^\epsilon X_j ) \<X_j, g_s^\epsilon X_1Y_0\> db_s\\
&+{1\over \epsilon} \int_0^t dF(\tilde x_s^\epsilon X_j ) \<X_j, g_s^\epsilon X_1^2Y_0\>ds.
\end{split}}.$$
Applying the identities $X_i^2=-I$ we obtain
\begin{equation}\label{Ito-tight}
{\begin{split}F(\tilde x_t^\epsilon)
&=F(u_0)-\epsilon \sum_{j=2}^3\left(dF(\tilde x_t^\epsilon X_j )\<X_j, g_t^\epsilon Y_0\> -dF(u_0X_j)\<X_j,Y_0\>\right)\\&
+\epsilon \sum_{j,k=2}^3\int_0^t \nabla^L dF(\tilde x_s^\epsilon X_k, \tilde x_s^\epsilon X_j)
\<X_j, g_s^\epsilon Y_0\>\<X_k, g_s^\epsilon Y_0\>ds\\
&+ \sqrt \epsilon \sum_{j=2}^3\int_0^t dF(\tilde x_s^\epsilon X_j ) \<X_j, g_s^\epsilon X_1Y_0\> db_s.
\end{split}}\end{equation}
Since  $F$ is a smooth function on compact manifolds, the probability distribution of $\{\tilde x_{t\over \epsilon}^\epsilon, \epsilon>0\}$ is tight, see Lemma \ref {Hopf-lemma-tight} below.  
We now move to the canonical probability space with the 
standard filtration $\F_t$.  By Lemma \ref{Hopf-convergence} below,  conditioning on the filtration $\F_s$, the canonical filtration  on the canonical probability space, 
$$\epsilon \sum_{j,k=2}^3   \int_{r\over \epsilon}^{t\over \epsilon}\nabla^L dF(\tilde x_s^\epsilon X_k, \tilde x_s^\epsilon X_j)
\<X_j, g_s^\epsilon Y_0\>\<X_k, g_s^\epsilon Y_0\>ds$$
converges to $$ \sum_{j,k=2}^3\int_r^t\nabla^L dF(uX_k, uX_j)  \int_{S^1} \<X_j, g Y_0\> \<X_k, g Y_0\> dg ds.$$
Here $dg$ is the Haar measure on $S^1$. It is easy to check that $$\int_{S^1} \<X_2, g Y_0\> \<X_3, g Y_0\> dg=0,$$
either by direct computation or note that there is $g'\in U(1)$ such that $g'X_2=-X_2$ and $g'X_3=X_3$ and using the
translation invariance of the Haar measure. Since there is an element of $S^1$ that maps $X_2$ to $X_3$, 
$$ \int_{S^1} \<X_2, g Y_0\>^2  dg=\int_{S^1} \<X_3, g Y_0\>^2 dg.$$
Note that 
 $$\sum_{j=1}^2 \int_{S^1} \<X_j, g Y_0\> \<X_j, g Y_0\> dg=|gY_0|^2=|Y_0|^2.$$
We conclude that  $\tilde x_{t\over \epsilon}^\epsilon$ converges in distribution and its law is determined by the generator $\bar \L F(u)={1\over 2}|Y_0|^2  \nabla^L dF(uX_2, uX_2)+{1\over 2} |Y_0|^2  \nabla^L dF(uX_3, uX_3)$. 
Since $\nabla^L X_i^*=0$ for $i=2,3$, $\sum_{i=2}^3 \nabla^L d(f\circ \pi)(Y_i^*,Y_i^*)=\trace \nabla df $. Note also the Riemannian metric on $S^2$ is that induced from $S^3$, the process $x_t^\epsilon$ has generator ${1\over 2}|Y_0|^2 \Delta_{S^2}$ and is a Brownian motion when $Y_0$ is a unit vector. 
\end{proof}

 \begin{lemma}
 \label{Hopf-lemma-tight}
Let  $\mu^\epsilon$ be the probability distributions of  the stochastic processes $(\tilde x^\epsilon_{t\over \epsilon}, t\ge 0 )$ from the theorem.
 Then $\{\mu^\epsilon, \epsilon>0\}$ is relatively compact. 
 \end{lemma}
 \begin{proof}
Write $y_t^\epsilon=\tilde x^\epsilon_{t\over \epsilon}$ for simplicity. 
Let $\mu_n$ be a subsequence from $\{\mu_\epsilon\}$ corresponding to a sequence of numbers $\epsilon_n$. We wish to prove that it has a weakly convergent subsequence.  It is sufficient to prove that the family of measures $\mu_n$ is tight, \ie for every $\delta>0$  there exists a compact set $K_\delta\subset M$ such that $\mu_n(K_\delta)>1-\delta$ for all $n$.  As  probability measures on the space of continuous paths on $M$,  $\mu_n(\sigma: \sigma(0)=y_0)=1$ where $\sigma:[0,1]\to M$ is a continuous path on $M$. 
For any $y_1, y_2\in M$,  let $\phi:M\times M\to \R$ be a smooth function that agrees with the Riemannian distance function when $d(y_1, y_2)<a/2$ where $a$ is the injectivity radius of $M$ and $\phi(y_1,y_2)=1$ when $d(y_1, y_2)>2a$. This is possible by taking $\phi=\alpha\circ d$ where $\alpha:\R_+\to \R$ is a suitable bump function with $\alpha$ the identity function on $[0,a/2]$. Then $\phi$ is a distance function on $M$ that generates the same topology as  $d$. The family of measures $\{
\mu_n\}$  is tight
if   for any  $a, \eta>0$ there exists $0<\delta<1$ such that there is an $\epsilon_0>0$, with
$$\P\left (\omega:  \sup_{|s-t|<\delta} \phi(y_s^{\epsilon_n}, y_t^{\epsilon_n})>a\right)<\eta, \quad \hbox{when }  \epsilon<\epsilon_0.$$
In the proof of the Theorem, take $F(y)=\phi(y,u)$. Then by formula (\ref{Ito-tight}),
$$\E \sup_{s\le \delta}\phi^2(y_s^{\epsilon}, u)
\le  \phi^2(y_0^\epsilon,u)+C\epsilon+ \epsilon \delta$$
for come constant $C$. Let $\phi_t^\epsilon(y, \omega))$ denote $y_\cdot^\epsilon(\omega)$ with $y_0^\epsilon(\omega)=y$. Let $\theta_s$ denotes the shift operator in the Wiener space. By the Cocycle property, for $s<t$, 
$${\begin{split}
\E \sup_{|s-t|\le \delta}\phi^2(y_s^{\epsilon_n}, y_t^{\epsilon_n})
&=\E  \E\{  \sup_{|s-t|\le \delta} \phi^2( z, \phi_{t-s}^\epsilon(z, \theta_{t-s}(\omega))) | y_s^{\epsilon_n}=z\}
\le C\epsilon+ \epsilon \delta
\end{split}}$$
and the required tightness holds.
 \end{proof}
 
 Let $y_t^n$ be a family of Markov processes on a Riemannian manifold $M$ that is relatively compact. Represent this as the coordinate process on path space
 with measure $\mu^n$, the distribution of $y_\cdot^\epsilon$. Suppose that $\mu_n$ converges weakly to $\bar \mu$.  Let $F: M\to \R$ be a smooth function with compact support. Let $\A$ be a diffusion operator. Suppose that
 $$\int f \left( F(X_t)-F(X_0)-\int_s^t\A F(X_r)dr\right) d\mu_n\to 0$$ for any function $f$ that is measurable with respect to $\F_s$ where $(\F_s, s\ge 0)$ is the canonical filtration. Then  $\bar \mu$  is the probability distribution of a $\A$-diffusion. In fact letting 
 $M_t^F=F(X_t)-F(X_0)-\int_s^t\A F(X_r)dr$, then $M_t^F$ is a $\mu$ martingale and $\mu$ is the solution to the martingale problem associated to $\A$.
 The following lemma reflects this philosophy. let $\epsilon_n$ be a sequence converges to zero. 
We are interested in the term  $$\epsilon \sum_{j,k=2}^3   \int_{r\over \epsilon}^{t\over \epsilon}\nabla^L dF(\tilde x_s^\epsilon X_k, \tilde x_s^\epsilon X_j)
\<X_j, g_s^\epsilon Y_0\>\<X_k, g_s^\epsilon Y_0\>ds
 . $$

\begin{lemma}\label{Hopf-convergence}
Let $(y_t^\epsilon,h_t^\epsilon)$ be a family of $SU(2)\times U(1)$ valued stochastic processes  on a probability space such that the law of $(y_t^\epsilon,h_t^\epsilon)$ agrees with that of $(\tilde x_{t\over \epsilon}^{\epsilon}, g_t^{\epsilon})$ in Theorem \ref{Hopf-theorem}. Let $(\tilde x_{t\over \epsilon_n}^{\epsilon_n}, g_t^{\epsilon_n})$ be a  weakly convergent sequence.  Let $y^n_t:= y_t^{\epsilon_n}$. We may assume that $y_\cdot ^n$ converges to $\bar y_\cdot $ almost surely.
Let $F: SU(2)\to \R$ be a smooth function. Define 
$${\begin{split} \bar \L F(u)&=\sum_{j,k=2}^3 \nabla^L dF(uX_k, uX_j)  \int_{S^1} \<X_j, g Y_0\> \<X_k, g Y_0\> dg,\\
\A^\epsilon F(y_s^\epsilon, g_s^\epsilon) &= \sum_{j,k=2}^3  \nabla^L dF(y^\epsilon_{s\epsilon} X_k, y_{s\epsilon} ^\epsilon X_j)
\<X_j, g_s^\epsilon Y_0\>\<X_k, g_s^\epsilon Y_0\>.
\end{split}} $$
Then  the following convergence holds in $L^1$,
$$\epsilon \int_{s\over \epsilon}^{t\over \epsilon} \A^\epsilon F(y_r^\epsilon, g_r^\epsilon) dr\to \int_s^t \bar \L F(y_r^\epsilon) dr,$$
and for any real valued bounded function $\phi$   on the path space,
$$\E \phi(y_r^n, r\le s) \left(  F(y_t^n)-F(y_s^n)-\int_s^t \bar \L F(y_r^n) dr \right) \to 0.$$
\end{lemma}
\begin{proof}
By formula (\ref{Ito-tight}),
$$ {\begin{split}
F(y_t^n)-F(y_s^n)-\int_s^t \bar \L F(y_r^n) dr =&\epsilon \sum_{j=2}^3 dF(y_{t\epsilon}^\epsilon X_j )\<X_j, g_t^\epsilon Y_0\> -\epsilon \sum_{j=2}^3 dF(y_{s\epsilon}^\epsilon X_j )\<X_j, g_s^\epsilon Y_0\>\\
& +\epsilon\int_{s\over \epsilon}^{t\over \epsilon} \A^n F(y_r^\epsilon, g_r^\epsilon) dr-\int_s^t \bar \L F(y_r^n) dr.\end{split}}$$
It is sufficient to prove that $$\epsilon\int_{s\over \epsilon}^{t\over \epsilon} \A^n F(y_r^\epsilon, g_r^\epsilon) dr\to \int_s^t \bar \L F(y_r^n) dr.$$
Let $t_0=s<t_1<\dots <t_n=t$ be a division of $[s,t]$ with appropriate scale. Let  $\Delta t_t=t_{i+1}-t_i$.  Assume that $\Delta t_i=\sqrt \epsilon$ so ${\Delta t_i\over \epsilon}={1\over \sqrt \epsilon}$ is large.  On each interval $[{t_i\over \epsilon}, {t_{i+1}\over \epsilon}]$,
$$| \nabla^L dF(y^\epsilon_{s\epsilon} X_k, y_{s\epsilon} ^\epsilon X_j)-
\nabla^L dF(y^\epsilon_{t_i\epsilon} X_k, y^\epsilon_{t_i \epsilon } X_j) |\le C|y^\epsilon_{s\epsilon}-y^\epsilon_{t_i\epsilon}| \sim o(\sqrt\epsilon \Delta t_i),$$
where $\sim$ means in the order of, after taking expectations. Let $dg$ be the Haar measure on $S^1$. It is the invariant measure for 
$g_{t\epsilon}^\epsilon$ where $g_t^\epsilon$ is solution to  $dg_t^\epsilon={1\over \epsilon} g_t^\epsilon X_1 dt$. By Birkhoff's  ergodic theorem on $S^1$ we obtain,
$${\begin{split} 
&\epsilon\int_{s\over \epsilon}^{t\over \epsilon} \nabla^L dF(y^\epsilon_{s\epsilon} X_k, y_{s\epsilon} ^\epsilon X_j)
\<X_j, g_s^\epsilon Y_0\>\<X_k, g_s^\epsilon Y_0\>ds\\
 &\sim  \sum_i  \Delta  t_i     \nabla^L dF(y^\epsilon_{t_i\epsilon} X_k, y^\epsilon_{t_i \epsilon } X_j) 
{\epsilon \over \Delta t_i } \int_{t_i\over \epsilon}^{t_{i+1}\over \epsilon}  \<X_j, g_s^\epsilon Y_0\>\<X_k, g_s^\epsilon Y_0\>ds\\
&\sim   \sum_i    \Delta  t_i   \nabla^L dF(y^\epsilon_{t_i\epsilon} X_k, y^\epsilon_{t_i \epsilon } X_j) 
\int_{S^1}  \<X_j, gY_0\>\<X_k, gY_0\>dg\\
&\sim \int_{s\over \epsilon}^{t\over \epsilon} \nabla^L dF(y^\epsilon_{r\epsilon} X_k, y^\epsilon_{r \epsilon } X_j) dr\int_{S^1}  \<X_j, gY_0\>\<X_k, gY_0\>dg\\
&\sim \int_s^t  \nabla^L dF(y^\epsilon_{r} X_k, y^\epsilon_{r  } X_j) dr\int_{S^1}  \<X_j, gY_0\>\<X_k, gY_0\>dg.
\end{split}} $$
\end{proof}

\section{Perturbed Systems On Principal  Bundles}\label{OM-section}
Let $M$ be a smooth finite dimensional manifold. 
The  frame bundle $\pi: FM\to M$ is a principal bundle with group action $GL(n,\R)$. Its total space is the collection of all linear isomorphisms $u:\R^n\to T_{\pi(u)}M$. Given a Riemannian metric on $M$, the orthonormal frame bundle $\pi: OM\to M$ is a reduced bundle with group action $O(n)$  and the fibre at $u:\R^n\to T_{\pi(u)}M$ consisting of isometric linear maps. 
 The total space of the frame bundle or  the orthonormal frame bundle is a manifold in its own right.
 If $M$ is oriented the orthonormal frame bundle consists of two components in which case we only need to consider the action by the component $SO(n)$ of the group that contains the identity. We assume that $n>1$. In all cases the group will be denoted by $G$, its Lie algebra by $\g$ and the right action  by $a\in G$ is denoted by $R_a$. In case of $G=SO(n)$,  $\g=\so(n)$ is the space of skew symmetric matrices. 

Denote by $TOM$ the tangent space of  $OM$ and by $VT_uOM$ the naturally defined vertical sub-bundle, $VT_uOM=\ker[T_u\pi]$.  
If $A$ belongs to the Lie algebra $\so(n)$, denote by  $A^*$  the  fundamental vertical vector field on $OM$ induced by right multiplication,
 $$A^*(u)={d\over dt}|_{t=0} u\exp(tA).$$ 
 Then an o.n.b of $\so(n)$ induces a family of vertical vector fields that spans $VTOM$ and $VTOM$ is an integrable sub-bundle of the tangent bundle .

A connection $\nabla$ on the tangent space of $M$ induces a splitting of the tangent spaces of $T_uOM$:
$$T_uOM=HT_uOM\oplus VT_uOM.$$
Let $HTOM=\sqcup_u HT_uOM$ and $VTOM=\sqcup_u  VT_uOM$. Then  $HTOM$ is a right invariant distribution and the splitting is
an Ehresmann connection of $TOM$.
Tangent vectors or vector fields are called horizontal (respectively vertical) if they take vales in $TOM$ (respectively in $VTOM$).
This determines a linear connection $\nabla$ on $M$, a connection 1-form $\varpi_u: T_uOM\to \so(n)$, 
and a horizontal lifting map $\h_u: T_{\pi(u)}M\to T_uOM$.

To each $e\in \R^n$, there is associated a  standard  (or basic) horizontal vector field  $H(e)$ given by 
$u\mapsto H(u)(e)\equiv \h_u(ue)$.
For $A\in \so(n)$, $[H(e),A^*]$ is a horizontal vector field.  For $e, \tilde e\in \R^n$   the vertical part of $[H(e),H(\tilde e)]$ is given by the curvature form $\Omega$, $[H(e),H(\tilde e)]=2\Omega(H(e),H(\tilde e))$.
Let $\{e_i\}$ be an orthonormal basis of $\R^n$ and define $H_i=H(u)(e_i)$.

Let $(\Omega, \F, \F_t,P)$ be a filtered probability space with the usual assumptions and let $\{ w_t^j,  b_t^l, 1\le j\le p, 1\le l\le m\}$ be independent one dimensional Brownian motions.
Let $ w_t=( w_t^1,\dots,  w_t^m)$  and $ b_t=( b_t^1,\dots,  b_t^p)$.
 Let   $\{\X_l, l=0,1,2,\dots, m\}$  be a family of horizontal vector fields and $\{Z_j, j=0,1,\dots p\}$  
a family of vertical vector fields.
Consider the SDE
\begin{equation}\label{averaging-0}\left\{
\begin{array}{ll}
du_t^\epsilon&=\sqrt \epsilon_1 \sum_{l=1}^m \X_l(u_t^\epsilon)\circ d b_t^l+\epsilon \X_0(u_t^\epsilon)dt+\sqrt \delta  \sum_{j=1}^p Z_j(u_t^\epsilon)\circ d w_t^{j}+\delta Z_0(u_t^\epsilon)dt,\\
u_0^\epsilon&=u_0.
\end{array}\right.
\end{equation} 
here $\epsilon_1, \epsilon $ and $\delta$ are parameters.
The infinitesimal generator of the SDE is ${1\over 2}\epsilon_1 \sum_{l=1}^m \X_l^2+\epsilon \X_0+\delta \sum_{j=1}^m Z_j^2+\delta Z_0$.

%
%
%

  \subsection{Perturbation to Vertical Flows}
  \label{section-vertical}
A Riemannian connection $\nabla$ on $M$ is a connection that is compatible with the Riemannian metric, with possibly a non-vanishing torsion $\T$. We take the horizontal bundle  on the principle bundle induced by this connection. We will assume that the connection is complete, i.e. every geodesic extends to all finite time parameter. This is so if every standard horizontal vector field is complete.
 Let $\varpi: T_uOM\to \so(n)$ be the connection 1-form, corresponding to the given Riemannian connection $\nabla$, which is determined by adjoint invariance and its values on fundamental vertical vector fields:
  $(R_g)^*\varpi=\ad(g^{-1})\varpi$ and $\varpi(A^*)\equiv A$.   
  Let  $\theta_u: T_uOM\to \R^n$ be the canonical 1-form such that $\theta_u (\h_u(ue))=e$.
 Let $\breve \nabla$ be the direct sum connection on $TOM$. For any vector  $v\in T_uOM$  
  and vector field $U$ on $OM$, 
  $$\breve\nabla_v U= \varpi^{-1}d(\varpi (U))(v)+\theta^{-1}d(\theta(U))(v).$$
  This connection $\breve \nabla$ has zero curvature and a non-vanishing torsion in general. A formula for the torsion of $\breve \nabla$ is given in 
  Li \cite{Li-OM-2}.

In  (\ref{averaging-0}) take $\epsilon_1=\epsilon$ and $\delta=1$.
Let $\tilde L^\epsilon=\L_0+\epsilon \L_1$ where
$$ \L_1={1\over 2}\sum_{l=1}^m L_{\X_l}L_{\X_l}+L_{\X_0}, \L_0={1\over 2}\sum_{j=1}^p L_{Z_j}L_{Z_j}+\L_{Z_0}.$$

\begin{theorem}
\label{level-thm}
Assume that $M$  has positive injectivity radius,  $\{ \varpi_u [ Z_j(u)]\}_{j=1}^m$ spans $\g$, and the vector fields  $\{\X_l, l\ge 0\}$  and $\{ |\breve\nabla _{\X_l} \X_l|, \l\ge 1\}$ have linear growth.  Let $u_t^\epsilon$ be a solution with initial value $u_0\in OM$, to the SDE
\begin{equation}\label{averaging}
du_t^\epsilon=\sqrt \epsilon \sum_{l=1}^m \X_l(u_t^\epsilon)\circ d b_t^l+\epsilon \X_0(u_t^\epsilon)dt+ \sum_{j=1}^p Z_j(u_t^\epsilon)\circ d w_t^{j}Z_0(u_t^\epsilon)dt.
\end{equation}
Let $x_t^\epsilon=\pi(u_t^\epsilon)$ and $\tilde x_t^\epsilon$ its horizontal lift. Let $\mu_u$  be the invariant measure of the following SDE on $G$:
 $$dg_t= \sum_{j=1}^mTL_{g_t} \varpi [ Z_{j}(ug_t)] \circ d w_t^{j}+TL_{g_t} \varpi [ Z_{0}(ug_t)]dt.$$
 Define $${\begin{split}
b(u)&=   \int_{G} \left({1\over 2} \sum_{l=1}^p \breve \nabla_{\X_l} \X_l(ug)+\X_0(ug)\right)d\mu_u(g)\\
a_{i,j}(u)&=\int_{G} \sum_{l=1}^p \<TR_g^{-1}\X_l(ug), H_i(u ) \>\<TR_g^{-1}\X_l(ug), H_j(u ) \>\;d\mu_u(g),
\end{split} }$$
 Then $\tilde x_{t\over\epsilon}^\epsilon$ converges weakly with limiting generator  $\bar \L$. For $F:OM\to \R$ smooth with compact support,
 $$\bar \L F(u) = dF(b(u))
   			+{1\over 2}\sum_{i,j=1}^p a_{i,j}(u)\breve \nabla dF(H_i(u),H_j(u) ).$$

 \end{theorem}

\begin{proof}
Since $\tilde x_t^\epsilon$ and $u_t^\epsilon$ belong to the same fibre we may define
 $g_t^\epsilon \in G$  by $u_t^\epsilon=\tilde x_t^\epsilon g^\epsilon_t$. If $a_t$ is a $C^1$ curve in $G$ 
 $$ {d\over dt}|_{t}ua_t={d\over dr}_{|_{r=0}} ua_t a_t^{-1}a_{r+t} =(a_t^{-1}\dot a_t)^*(ua_t).$$
It follows that
$$du_t^\epsilon=TR_{g_t^\epsilon} d\tilde x_t^\epsilon+ (TL_{(g_t^\epsilon)^{-1}}dg_t^\epsilon)^*(u_t^\epsilon).$$
Since right translation of horizontal vectors are horizontal, $\varpi(du_t^\epsilon) =TL_{(g_t^\epsilon)^{-1}}dg_t^\epsilon$ and
\begin{equation}
\label{left-BM}
dg_t^\epsilon= \sum_{j=1}^mTL_{g_t^\epsilon} \varpi [ Z_{j}(\tilde x_t^\epsilon g_t^\epsilon)] \circ dw_t^{j}+TL_{g_t^\epsilon} \varpi [ Z_{0}(\tilde x_t^\epsilon g_t^\epsilon)]dt.
\end{equation}
By It\^o's formula, 
 $dx_t^\epsilon= \sqrt\epsilon\sum_{l=1}^p T\pi(\X_l (u_t^\epsilon))\circ d b_t^l+\epsilon T\pi(\X_0 (u_t^\epsilon))dt$ so
$$d \tilde x_t^\epsilon= \h_{\tilde x_t}(\circ d x_t^\epsilon)=
{\sqrt\epsilon} \,\sum_{l=1}^p \h_{\tilde x_t^\epsilon}[T\pi(\X_l(u_t^\epsilon))]\circ d b_t^l+\epsilon\h_{\tilde x_t^\epsilon}[T\pi(\X_0(u_t^\epsilon))]dt.$$
 By assumption on the vector fields $\X_l$,  the above SDE does not explode and $\pi(u_t^\epsilon)$ exists for all time. In terms of the group action, we have
\begin{equation}
d \tilde x_t^\epsilon=
{\sqrt\epsilon} \,\sum_{l=1}^p \h_{\tilde x_t^\epsilon}[T_{\tilde x_t^\epsilon g_t^\epsilon}\pi(\X_l(\tilde x_t^\epsilon g_t^\epsilon)]\circ d b_t^l+\epsilon \h_{\tilde x_t^\epsilon}[T_{\tilde x_t^\epsilon g_t^\epsilon}\pi(\X_0(\tilde x_t^\epsilon g_t^\epsilon)]dt.
\end{equation}
Let $\mu^\epsilon$ be the laws of the $\tilde x_t^\epsilon$.
We first show that $\{\mu^\epsilon\}$ is tight. By Prohorov's theorem  a family of probability measures is tight if it is relatively compact. Since $\tilde x_0^\epsilon=u_0$ it suffices  to estimate the modulus of continuity and show that for all positive numbers $a, \eta$, there exists $\delta>0$ such that for all $\epsilon$ reasonably small,  see Billingsley \cite{Billingsley-68} Ethier-Kurtz\cite{Ethier-Kurtz86},
$$P(\omega: \sup_{|s-t|<\delta} d(\tilde x_t^\epsilon, \tilde x_s^\epsilon)>a)<\eta.$$
 Here $d$ denotes a distance function on $OM$.   The Riemannian distance function is not smooth on the cut locus. The cut locus of $OM$ is in general not predictable by that of $M$. To avoid any assumption on the cut locus
of $OM$ we construct a new distance function that preserves the topology of $OM$.

Let $x\in M$ and $2a$ the minimum of $1$ and the injectivity radius of $M$.  
Let $\phi: \R_+\to \R_+$ be a smooth concave function such that $\phi(r)=r$ when $r<a$ and $\phi(r)=1$ when $r\ge 2a$, e.g. $\phi$ is the convolution of $\min(1, r)$ with a standard mollifier supported in the set $\{r: |r-{3a\over 2}|<a/2\}$. 
 Let $\rho$ and $\tilde \rho$ be respectively  the Riemannian distance on $M$ and $OM$.  Then $\phi \circ \rho$ and $d:=\phi\circ  \tilde\rho$ are distance functions.  
For $u\in \pi^{-1}(x)$,
$${ \begin{split}
\phi \circ \tilde \rho(u, \tilde x_t^\epsilon)
=&( \phi\circ \tilde \rho)(u, \tilde x_0^\epsilon) +\int_0^t  d(\phi\circ \tilde \rho) \left({\sqrt\epsilon} \,\sum_{l=1}^p \h_{\tilde x_s^\epsilon}[T\pi(\X_l(u_s^\epsilon))]\circ dB_s^l\right)\\
&+\int_0^t \epsilon \;d(\phi\circ \tilde \rho)\;\h_{\tilde x_s^\epsilon}[T\pi(\X_0(u_s^\epsilon))]\;ds
\end{split}}$$
$${ \begin{split}=& ( \phi\circ \tilde \rho )(u, \tilde x_0^\epsilon)+ \int_0^t d(\phi\circ \rho) \left({\sqrt\epsilon} \sum_{l=1}^p [T\pi(\X_l(u_s^\epsilon))]dB_s^l\right)\\
&+ \epsilon \sum_{l=1}^p\int_0^t  \nabla d(\phi\circ  \rho) \left(  T\pi(\X_l(u_s^\epsilon)),   T\pi(\X_l(u_s^\epsilon))\right)\; ds\\
&+ \epsilon  \int_0^t   d(\phi\circ  \rho) \left({1\over 2} \sum_{l=1}^p \nabla_{ T\pi(\X_l)} (T\pi \circ \X_l)(u_s^\epsilon)+T\pi(\X_0(u_s^\epsilon)) \right)ds.
\end{split}}$$
Since $\phi\circ \rho$ has compact support and the vector fields concerned have linear growth, 
$|T\pi(\X_l(u_s^\epsilon))|\le C(1+\rho(u_s^\epsilon, u))\le [C+C\tilde\rho(\tilde x_s^\epsilon, \tilde u_s^\epsilon)]+C\tilde\rho(u, \tilde x_s^\epsilon)$ some $u\in OM$. The quantity $C+C\tilde\rho(\tilde x_s^\epsilon, \tilde u_s^\epsilon)$ is bounded from the compactness of $G$ and it follows that  $\E[ \phi\circ \tilde \rho(u,\tilde x_t^\epsilon)]^2
)\le C_1(t) ( ( \phi\circ \tilde \rho)^2(u, \tilde x_0^\epsilon) +\epsilon t)$
for some constant $C$.  By the Markov property and the estimates below the required tightness follows,
$$\E[ \phi\circ \tilde \rho(\tilde x_{s\over \epsilon}^\epsilon,\tilde x_{t\over \epsilon}^\epsilon)]^2)
\le C_1 |t-s|.$$

%

By the right invariance of the horizontal lift,
$$\h_{\tilde x_s^\epsilon}[T_{\tilde x_s^\epsilon g_s^\epsilon}\pi(\X_l(\tilde x_s^\epsilon g_s^\epsilon)]
=TR_{(g_s^\epsilon)^{-1}}\X_l(u_s^\epsilon).$$
 Let $F:OM\to \R$ be a smooth function.  For $\breve \nabla$, the canonical direct sum connection on $OM$ associated to $\nabla$,
\begin{eqnarray*}
F(\tilde x_t^\epsilon)&=&F(u_0)+{\sqrt\epsilon} \,\sum_{l=1}^p\int_0^t dF\left( TR_{(g_s^\epsilon)^{-1}}\X_l(u_s^\epsilon) \right)dB_s^l\\
&&+{1\over 2}\epsilon \sum_{l=1}^p
\int_0^t \breve\nabla dF\left(TR_{(g_s^\epsilon)^{-1}}\X_l(u_s^\epsilon), TR_{(g_s^\epsilon)^{-1}}\X_l(u_s^\epsilon) \right)ds\\
&&+{1\over 2}\epsilon \sum_{l=1}^p  \int_0^t  dF\left(\breve \nabla_{\X_l} \X_l(u_s^\epsilon)+\X_0(u_s^\epsilon) \right) ds.
\end{eqnarray*}
By tightness  and Prohorov's theorem we may take a sequence $\epsilon_n\to 0$ with the property that $\tilde x_{t\over \epsilon}^{\epsilon_n}$ converges in law to a probability measure $\mu$.  Let $X_\cdot$ be the coordinate process on the path space and $\G_t=\sigma\{X_s: 0\le s\le t\}$. Since $G$ is compact the following term has at most quadratic growth,
$$\int_{G} \<TR_{g^{-1}}\X_l(u  g), H_i(u)\>\< TR_{g^{-1}}\X_l(u  g), H_j(u)\>\mu_u(dg),$$
 and by the same argument $\int_{G} \sum_{l=1}^p \breve \nabla \X_l(\X_l)(u g) \mu_u(dg) $ has linear growth. 
To identify the limiting process it suffices to show that for all real-valued smooth functions $F$ on $OM$ with compact support, $$\int \left(F(X_t)-F(X_s)-\int_s^t \bar \L F(X_r)dr\right) g\;d\mu^\epsilon$$ converges to zero
where $g$ is any real-valued bounded $\G_s$-measurable function on the Wiener space and $X_t$ the canonical process.
 
Let $z_t^n$ be a sequence of random variables  whose law agrees with that of $\tilde x_{t\over \epsilon_n}^{\epsilon_n}$ for some sequence $\epsilon_n$ and $z_t^n$ converges almost surely. Let $g$ be a $\{z_s^n, s\le t \}$-adapted bounded function. For $t\ge s$, 
$$\E g\left(F(z_{t}^n)-F(z_{s}^n)
-\int_s^t \bar \L F(z_{r}^n)dr\right)=\E \left[g \int_s^t  (\A^{\epsilon_n} F-\bar \L F) (z_{r}^n) dr\right] \to 0,$$ where $\A^{\epsilon_n}F$ is given by the bounded variation part in the formula for $F(\tilde x_t^\epsilon)$. The convergence holds since $G$ is compact and also the invariant measure $\mu_{G}$ for the 
elliptic SDE (\ref{left-BM}) is ergodic.  The proof is standard and follows from the Lemma below. See e.g.   Hasminskii \cite{Hasminskii68},  Papanicolaou-Stroock-Varadhan \cite{Papanicolaou-Stroock-Varadhan77}. \end{proof}

\begin{lemma}\label{lemma}
 Let $f$ be a bounded function with bounded derivative then
 $$\int_s^t   \A^\epsilon f (\tilde x_{r\over \epsilon}^\epsilon) dr
 =\int_s^t   \bar\L f (\tilde x_{r\over \epsilon}^\epsilon) dr +R(f, \epsilon,s, t)$$  where
 $\left(\E \sup_{s\le t}  |R(f, \epsilon, s,t)|^\beta \right)^{1\over \beta} \le C(t)\epsilon^{1\over 3}$ for any $\beta>1$.
 \end{lemma}
The proof is completely analogous of that  of Lemma 3.2 in \cite{Li-averaging}.
 In sub-intervals whose length is very small compared to $1/\epsilon$ we consider $\tilde x_s^\epsilon$ as constants, and apply the ergodic theorem on each interval. With the size of the sub-intervals chosen correctly, the sum over all sub-intervals of the limits forms a Riemann sum. The convergence follows from the Cocycle property of the flows, estimates for the rate of convergence in the ergodic theorem and the regularity of the function $\A^\epsilon f$.

\medskip

  In the theorem above the assumption on the injectivity radius can be removed in the case of the projection being a Brownian motion with bounded drift. See e.g. the estimates in \cite{Li-flow}.  We look into two special cases, when the horizontal vector fields are either  right invariant (lifts of vector fields on the manifold $M$) or the standard horizontal vector fields.

 \begin{example} [The Right Invariant Case]  
 \label{example-right} Let $X_l, l=0,1,2,\dots m$, be vector fields on $M$.
 Define  $\X_l(u)=\h_u(X_l(\pi(u)))$  and we have
$$ du_t^\epsilon=\sqrt \epsilon \sum_{l=1}^p \X_l(u_t^\epsilon)\circ d b_t^l+\epsilon \X_0(u_t^\epsilon)dt+ \sum_{j=1}^m  Z_j(u_t^\epsilon)\circ d w_t^{j}+ Z_0(u_t^\epsilon)dt.$$
 The projection $\pi(u_t^\epsilon) $ satisfies
$d x_t^\epsilon
=\sqrt \epsilon \sum_{l=1}^p X_l(x_t^\epsilon)\circ d b_t^l+\epsilon X_0(x_t^\epsilon)dt$.
 For all $\epsilon$, $x_{t\over \epsilon}^\epsilon$ are ${1\over 2}\sum L_{X_i}L_{X_i}+L_{X_0}$-diffusions.   The horizontal lift $\tilde x_{t\over \epsilon}$ of $x_{t\over \epsilon}$ are ${1\over 2}\sum L_{\X_i}L_{\X_i}+L_{\X_0}$-diffusions.  \end{example}

\begin{example} [The Rotational Invariant Case]
\label{example-2.2}
 Let $\{e_l\}_{l=1}^n$ be an o.n.b. of $\R^n$, $e_0\in \R^n$. Let $H_l(u)\equiv H(u)(e_l)$, and $H_0(u)\equiv H(u)(e_0)$ be horizontal vector fields. We have
 $$ du_t^\epsilon=\sqrt \epsilon \sum_{l=1}^n H_l(u_t^\epsilon)\circ d b_t^l+\epsilon H_0(u_t^\epsilon)dt+\sum_{j=1}^m  Z_j(u_t^\epsilon)\circ d w_t^{j}+ Z_0(u_t^\epsilon)dt.$$
Write $\tilde x_t^\epsilon=u_t^\epsilon(g_t^\epsilon)^{-1}$. Then
\begin{equation}\label{horizontal}
d\tilde x_t^\epsilon=\sqrt \epsilon H(\tilde x_t^\epsilon)( g_t^\epsilon \circ db_t)+\epsilon H(\tilde x_t^\epsilon)(g_t^\epsilon e_0)dt.
\end{equation}
Its `formal'   Stratonovitch correction term vanishes. 
If $\tilde x_0^\epsilon=u_0$ then $g_0^\epsilon$ is the identity matrix. Write $dw_t^0=dt$ and let $Z_j=\sigma_k^jA_j^*$ where $\{A_j\}$ is an o.n.b of $\so(n)$. Then
$$dg_t^\epsilon=\sum_{j,k} \sigma_k^j(u_t^\epsilon)g_t^\epsilon A_j\circ dw_t^k.$$

If  $\sigma_k^j$ are constants the process $g_t^\epsilon$ is independent of $\epsilon$
and $\tilde x_t^\epsilon$ is a Markov process on $OM$. If furthermore $e_0=0$, the law of $\tilde x_{t\over \epsilon}^\epsilon$, and hence that of  $x_{t\over \epsilon}^\epsilon$,  is independent of $\epsilon$. 
 This follows from the independence of $g_t^\epsilon$ and $\{b_t\}$.
Finally $\tilde x_t^\epsilon$ is a horizontal Brownian motion with projection $x_t$ a Markov process and a Brownian motion on $M$.  This is the construction of Brownian motions of Eells-Elworthy  \cite{Eells-Elworthy}. The invariance is no longer true for $e_0\not =0$. 
 
Remark: More generally if $\{\Phi_t(u)\}$ is a family of Markov processes on $OM$ with the property that  
 $\Phi_t(ug)\stackrel{law}{=} \Phi_t(u) \psi_t(g)$ for some $\psi_t(g)\in G$ and
 $\sigma\{\pi(\Phi_r(u))| r\le s\}=\sigma\{\Phi_r(u): r\le s\}$, then $\pi(\Phi_t(u))$ is a Markov process. Denote by $Q_t(u_0, du)$ the law of $\Phi_t(u_0)$ and let $f: M\to \R$ be a Borel measurable function, $x_t=\pi(\Phi_t(u_0))$, \begin{eqnarray*}
\E\{f(x_t)|\sigma\{x_r, r\le s\} \}
=\int (f\circ \pi)(u) Q_{t-s}(\tilde x_s, du).
\end{eqnarray*}
It follows that  $\int (f\circ \pi)(u) Q_{t-s}(\tilde x_s, du)=\int (f\circ \pi)(u\psi_s(g)) Q_{t-s}(\tilde x_sg, du)=\int (f\circ \pi)(u) Q_{t-s}(\tilde x_s g, du)$. So  $\E\{f(x_t)|\sigma\{x_r, r\le s\} \}$ depends only on $x_s=\pi(\tilde x_s)$. When $e_0=0$, the flow of (\ref{horizontal}) satisfies the rotational invariance condition and
 the horizontal lift of $x_t$ is a function of the path  $(x_r, r\le t ) $.

\end{example}

\begin{example}
Let $\alpha: M\times \R^n \to  \R^n$ be a smooth  map so that
$ \alpha(x)\in {\mathbb L}(\R^n; \R^n)$. Let$\{e_i\}_{i=1}^n$ be an o.n.b. of $\R^n$, $e_0\in \R^n$. Consider
$$du_t^\epsilon=\sqrt \epsilon \sum_{l=1}^n \h_u [\alpha(\pi(u))e_l]\circ d b_t^l+\epsilon \h_u [\alpha(\pi(u))e_0](u_t^\epsilon)dt+ \sum_{j=1}^m  Z_j(u_t^\epsilon)\circ d w_t^{j}+ Z_0(u_t^\epsilon)dt.$$
The projection $x_t^\epsilon=\pi(u_t^\epsilon)$ satisfies:
$$dx_t^\epsilon
=\sqrt \epsilon \sum_{l=1}^n  u_t^\epsilon \alpha(x_t^\epsilon)(e_l)\circ d b_t^l+\epsilon u_t^\epsilon \alpha(x_t^\epsilon)(e_0)dt=\sqrt \epsilon u_t^\epsilon \alpha(x_t^\epsilon)\circ d b_t+\epsilon u_t^\epsilon \alpha(x_t^\epsilon)(e_0)dt.$$
Let $\tilde x_t^\epsilon$ be the horizontal lifting map of $x_t^\epsilon$ and $g_t^\epsilon$ be an element of $G$ determined by $u_t^\epsilon=x_t^\epsilon g_t^\epsilon$. Then
$d\tilde x_t^\epsilon=\sqrt \epsilon  H(\tilde x_t^\epsilon) g_t^\epsilon \alpha(x_t^\epsilon)\circ d b_t+ \epsilon H(\tilde x_t^\epsilon)g_t^\epsilon \alpha(x_t^\epsilon)(e_0)dt$.
When $\alpha(x)$ is not trivial  the bounded variation term
 for $f(x_t)$, where $f:M\to \R$ is a smooth function, will involve $\sum_i\nabla df(u_t^\epsilon \alpha(x_t^\epsilon)e_i, u_t^\epsilon \alpha(x_t^\epsilon)e_i)$ which
 is no longer a trace. It will also involve the derivative of $\X_l$.
In this case it is useful to consider the system as  perturbation of the vertical SDE about which we know a lot more.  
\end{example}

\subsection{Perturbation of Ornstein-Uhlenbeck Type}
\label{OU-section}

We now describe the relation between  horizontal equations on frame bundles and geodesic flows. Let $P=GL(M)$ be the linear frame bundle over $M$.
A vector field on a frame bundle can be considered as a second order differential equation on the underlying manifold as below.   Fix $e_0\in \R^n$ and $H$ the isotropy group at $e_0$ of the action $G=GL(n,\R)$  on $\R^n$.  The tangent bundle $TM$ can be considered as a fibre bundle associated with the principal fibre bundle $P$ with fibre $\R^n$.  
The total space $E$ is $P\times \R^n/\sim$ where the equivalent class is determined by $[u, e]\sim [ug^{-1}, ge]$, any $g\in G$.   Elements of the form $ug$ where $g\in H$ belong to the same equivalence class.  It can be identified with the quotient bundle $P/H$, whose element containing $u$  is the equivalence class of the form $\{ug, g\in H\}$.  Denote by $\xi_0$ the coset  $H$. Let $\alpha$ be the associated map:
$$\alpha_{e_0}: u\in P\to  ue_0\in TM.$$
This induces a map $w\in T_uP\to T_{u}\alpha_{e_0} (w) \in T_{ue_0}TM$. Each element  $v\in TM$ has a representation $v=ue_0$, where $u$ is unique up to right translation by elements of $H$. Furthermore a right invariant vector field $W$ on $P$  induces a vector field on $TM$. In fact if $v=ue_0'=ue_0$ there is $g\in G$ with $u'=ug$ and $e_0'=g^{-1}e_0$. Since $\alpha_{e_0}(u)=\alpha_{e_0'}(R_gu)$,
$$T_u\alpha_{e_0}(W(u))=T_{u'}\alpha_{e_0'}TR_g(W(u))=T_{u'}\alpha_{e_0'}W(u').$$
This map $W\in \Gamma TP \mapsto X_W\in \Gamma TTM$ is independent of the choices of $e_0$. Fix $e_0$. Any vector field $W$ that is invariant by right translations of elements of $H$ induces a vector field on $TM$.
Consider a horizontal distribution determined by a connection on $TM$ and let $W(u)=H_u(e_0)$ be the fundamental horizontal vector field associated to $e_0$, 
the induced vector field is a geodesic spray $X$, i.e. in local co-ordinates $X(x,v)=(x,v, v, Z(x,v))$ and $Z(x,sv)=s^2Z(s,v)$, which corresponds to the geodesic flow equation on $TM$:
$$dv_t=-\Gamma_{\sigma_t}(v_t) (v_t),\; \dot \sigma_t=v_t, \sigma(0)=\pi(u), v(0)=ue_0.$$
Here $\Gamma$ denotes the Christoffel symbol. The corresponding horizontal flow on $P$ is given by $\dot u_t=H(u_t)(e_0)$.

\bigskip

Based on E. Nelson's Ornstein-Uhlenbeck theory of Brownian motions \cite{Nelson} we ask the following question.  What happens if we replace the driving Brownian motion $dw_t$ by  $v_t dt$ where $v_t$ is an Ornstein-Uhlenbeck process? Consider the position process $z_t$ in  $\R^n$ with velocity process satisfy the Langevin equation:
$${\begin{split}
dv_t^\epsilon&=-{1\over \epsilon}v_t^\epsilon dt+{1\over \epsilon}dw_t\\
\dot z_t^\epsilon&=v_t^\epsilon.
\end{split}}$$
where $w_t$ is a Brownian motion with values in $\R^n$ and $z_0=0$. The $z_t^\epsilon$ process converges to $w_t$ as $\epsilon \to 0$. The convergence holds almost surely and in fact the result holds if  $w_t$ is replaced by any continuous function.
We now interpret  the convergence  in terms of  homogenisation. First we rescale  the variables in space and time and setting $\tilde v_t=\sqrt \epsilon v_t, \tilde z_t=\sqrt \epsilon z_t$.  It is easy to see that $\tilde z_t^\epsilon$ is the slow variable and
$\sqrt\epsilon z_{t\over \epsilon}$ converges to a Brownian motion. In fact
$${\begin{split}
d\tilde v_t^\epsilon=-{1\over \epsilon}\tilde v_t^\epsilon dt+{1\over \sqrt \epsilon}dw_t, \qquad
\dot {\tilde z}_t^\epsilon=\tilde v_t^\epsilon.
\end{split}}$$

We must take care to take  this model to the orthonormal bundle.  First  we are not allowed to rescale  variables in non-linear spaces.  We should not  rescale the frame variable, in the orthonormal frame bundle, in space either. 

We shall consider the orthonormal frame bundle of an $n$-dimensional manifold with group action by $G=O(n)$ or $G=SO(n)$ if the manifold is oriented so that $OM$ is a connected manifold of its own right. In the latter case we assume that $n>1$. 
 Let $e_0 \in \R^n$ be a unit vector and $\{A_k, k=1,2,\dots, N=n(n-1)/2\}$ be elements of $\g$, and $A_0\in \g$.
Let $A_k^*$ be the corresponding fundamental vertical vector field corresponding to $A_k$. Consider  \begin{equation*}
du_t^\epsilon= H(u_t^\epsilon)(e_0)dt
+{1\over \sqrt \epsilon} \sum_{k=1}^NA_k^*(u_t^\epsilon)\circ dw_t^k+ {1\over \epsilon} A_0^*(u_t^\epsilon) dt.
\end{equation*}
For `$\epsilon=\infty$', the equation can be considered as the `geodesic flow' equation, as explained earlier.
 If $x_t^\epsilon=\pi(u_t^\epsilon)$ then $\dot x_t^\epsilon =u_t^\epsilon e_0$. Note that the change of the velocity of the motion on $M$ is always unitary. Due to the fast rotation, the geodesic has rapid changing directions and we expect to see a jittering motion and indeed we obtain a scaled Brownian motion in the limit if the rotational motion is elliptic. 


A related theorem is given in Dowell \cite{Dowell} stating that an Ornstein-Uhlenbeck position process  on 2-uniformly smooth Banach manifolds converges. Those are manifolds modelled on 2-uniformly smooth Banach spaces. By a 2-uniformly smooth Banach space $B$ we mean one with the property that there is a constant $C>0$ such that $||x+y||^2+||x-y||^2\le 2 ||x||^2+C||y||^2$ holds
for all $x, y\in B$. The iterated Ornstein-Uhlenbeck processes in \cite{Dowell} and
the settings of  manifolds with Lorenzo metrics, see e.g. Bailleul \cite{Bailleul}, are also worth exploring.     We expect interesting results  arise  for processes with infinite-dimensional noise.  For a related work, central limit theorem for geodesic flows, we refer to Enriquez-Franchi-LeJan \cite{Enriquez-Franchi-LeJan}.

\medskip

  Let $\A_k^*$ denote also the left invariant vector field on $G$ induced by $A_k\in g$:  $\A_k^*(g)=gA_k$. Let $\A={1\over 2}\sum_k (A_k^*)^2+A_0^*$. Denote by $\Delta^L$ the left invariant Laplacian on $G$.
For all $\epsilon$, let $u_0^\epsilon=u_0$ and $x_0=\pi(u_0)$. Let $\{e_i\}$ be an orthonormal basis of $\R^n$ and we may take $e_1=e_0$.
\begin{theorem}
\label{geodesic-flow}
Let $M$ be a compact Riemannian manifold, $u_0\in OM$. Let $u_t^\epsilon$ be the solution to the SDE on $OM$:
\begin{equation}\label{ou-1}du_t^\epsilon= H(u_t^\epsilon)(e_0)dt
+{1\over \sqrt \epsilon} \sum_{k=1}^NA_k^*(u_t^\epsilon)\circ dw_t^k+ {1\over \epsilon} A_0^*(u_t^\epsilon) dt, \; u_0^\epsilon=u_0.
\end{equation}
If $\A$ is elliptic then $\pi(u_{t\over \epsilon}^\epsilon)$  and its horizontal lift converges in law. If furthermore  $\A={1\over 2}\Delta^L$ then $\pi(u_{t\over \epsilon}^\epsilon)$ converges in law to a rescaled Brownian motion with generator
${4\over n(n-1)}\Delta$. Its horizontal lift converges in law to the diffusion process on $OM$ with generator
${4\over n(n-1)}\Delta_H $.
\end{theorem}
\begin{proof}
Define the Lie group valued process $g_t^\epsilon$ by $u_t^\epsilon=\tilde x_t^\epsilon g_t^\epsilon$ and $g_0^\epsilon=I$, the unit matrix. Following earlier computations and using the fact that $A^*$ is right invariant and $(R_g)^*\varpi=\ad(g^{-1})\varpi$,
$$dx_t^\epsilon= \tilde x_t^\epsilon g_t^\epsilon e_0dt, \quad d\tilde x_t= H(\tilde x_t^\epsilon) (g_t^\epsilon e_0)dt, \quad dg_t^\epsilon ={1\over \sqrt \epsilon} \sum_{k=1}^Ng_t^\epsilon A_k\circ dw_t^k +{1\over \epsilon} g_t^\epsilon A_0 dt.$$
For any smooth function $F: OM\to \R$ with compact support, 
$$F(\tilde x_t^\epsilon)=F(u_0)+\int_0^t dF\left(\tilde x_s^\epsilon) (H(\tilde x_s^\epsilon\right) g_s^\epsilon e_0)\,ds.
$$
As in the proof of the previous theorem, we see that the family $\{\tilde x_{t\over \epsilon}^\epsilon\}$ is tight and that it converges in law as $\epsilon\to 0$.   For each $u\in OM$ there is a solution $h_i:G\to \R$ to the equation 
 $$\A h_i(u,g)=dF(u)\left(H(u)e_i\right)\<ge_0, e_i\>.$$
For  $n>1$, $\int gOdg=\int gdg$ where $dg$ is the Haar measure normalised to be a probability measure and $O$ any matrix in $G$. The integral of $ge_0$ with respect to the Haar measure on $G$ vanishes. For $G=SO(n)$ it follows also
  from   $\int_{SO(n)} ge_0dg=\int_{S^{n-1}}s ds$, see Proposition 3.2.1 in Krantz-Parks \cite{Krantz-Parks}.
Denoting by $D_1h_i$ and $D_2h_i$ the differential of $h_i$ with respect to the first and the second variable respectively,
$${\begin{split}h_i(\tilde x_t^\epsilon, g_t^\epsilon)=&h(u_0,I)+{1\over \sqrt\epsilon} \sum_k \int_0^t (D_2h_i)_{(\tilde x_s^\epsilon, g_s^\epsilon)}(g_s^\epsilon A_k) dw_s^k
+{1\over \epsilon} \int_0^t \A h_i(\tilde x_s^\epsilon, g_s^\epsilon)ds \\
&+ \int_0^t (D_1h_i)_{(\tilde x_s^\epsilon, g_s^\epsilon)}\left(H(u_s^\epsilon)  g_s^\epsilon e_0\right)ds.
\end{split}}$$
Plug this back to $F(\tilde x_t^\epsilon)$ to see that
$${\begin{split}
F(\tilde x_{t\over \epsilon}^\epsilon)=&F(u_0)+ \epsilon \sum_i \left(h_i(\tilde x_{t\over \epsilon}^\epsilon, g_{t\over \epsilon}^\epsilon)-h(u_0,I)\right)
-\sqrt \epsilon \sum_{k,i} \int_0^{t\over \epsilon} (D_2h_i)_{(\tilde x_s^\epsilon, g_s^\epsilon)}(A_kg_s^\epsilon) dw_s^k\\
&-\epsilon\sum_i \int_0^{t\over \epsilon} (D_1h_i)_{(\tilde x_s^\epsilon, g_s^\epsilon)}(H(u_s^\epsilon)  g_s^\epsilon e_0)ds.
\end{split}}$$
The tightness of the law of $\{\tilde x_{t\over \epsilon}^\epsilon\}$ can be proved similar to that of the previous theorems, c.f. Lemma \ref{Hopf-lemma-tight}. Furthermore for $s<t$ the following convergence holds in $L^1$,  
$${\begin{split}
&-\epsilon \int_{s\over \epsilon}^{t\over \epsilon} (D_1h_i)_{(\tilde x_s^\epsilon, g_s^\epsilon)}(H(u_s^\epsilon)  g_s^\epsilon e_0)ds\\
&\to -\int_s^t \nabla dF  \left(H(\tilde x_s^\epsilon)e_j, H(\tilde x_s^\epsilon)e_i\right) ds\int_G  \<ge_0, e_j\> \A^{-1}\<ge_0, e_i\>d\mu(g),
\end{split}}$$
where $\mu$ is the unique invariant measure for the $\A$ diffusion on $G$. The proof is similar to that of Lemma \ref{Hopf-convergence}, taking into account of the following computation $${\begin{split} 
-(D_1h_i)_{(u,g)}\left(H(u)(ge_0)\right)&=-\nabla dF  \left(H(u)ge_0, H(u)e_i\right)\A^{-1}\<ge_0, e_i\>\\
&=-\sum_j\nabla dF  \left(H(u)e_j, H(u)e_i\right) \<ge_0, e_j\> \A^{-1}\<ge_0, e_i\>. \end{split}} $$

Now we assume that $\A={1\over 2}\Delta^L$. We may assume that $\{A_k\}$ is an orthonormal basis of $\g$ and $A_0=0$ and let
 $$a_{i,j}=-\int_G\<ge_0, e_j\> ({1\over 2}\Delta^L)^{-1}\<ge_0, e_i\>dg,$$
 where $g$ is the Haar measure on $G$.
For $i\not =j $ the cross term $a_{i,j}$ vanishes.  There is an element $O\in G$ such that
$Oe_i=-e_i$ and $Oe_j=e_j$.  Furthermore $\sum_l A_l^2=-{n-1\over 2}I$ and  $\sum_l gA_l^2=-{n-1\over 2} gI$.
It is easy to see that $({1\over 2}\Delta^L)^{-1}\<ge_0, e_i\>=-{ 4 \over n-1}\<ge_0, e_i\>$.  The integral $\int \<ge_0, e_i\>^2 dg$ is independent of $i$. In fact
$\int \<ge_0, e_i\>^2 dg=\int \<ge_0, Oe_i\>^2 dg $, for any $O\in G$. Since  $\int\sum_i \<ge_0, e_i\>^2 dg=1$ it follows that
$$ \qquad a_{i,i}={ 4\over n-1 } \int_G\<ge_0, e_i\>^2 \; dg ={4\over (n-1)n}.$$
Finally we see that
$${\begin{split} &\sum_{i,j}\nabla dF  \left(H(u)e_j, H(u)e_i\right) \<ge_0, e_j\> \A^{-1}\<ge_0, e_i\>\\
&={4\over (n-1)n}\sum_{i=1}^n   \nabla dF(\h_u(ue_i), \h_u(ue_i))={4\over (n-1)n}\Delta_H.
\end{split}}$$
The two operators $\Delta_H$ and $\Delta$ are intertwined by $\pi$, and $x^\epsilon_{t\over \epsilon}$ converges to a Brownian motion.
\end{proof}

\subsection{Another Intertwined Pair}

At this point we discuss  a question asked to  me by  J. Norris. 
Since the process on the orthonormal frame bundle encodes the Riemannian metric we  expect to see the Riemannian metric manifesting itself in some form, e.g. in the form of the corresponding Laplacian operator.  Does the system below have a non-degenerate limit which is not necessarily associated to the given Riemannian metric on $M$?  In general the intertwined system would look like the following, 
$${\begin{split}
du_t^\epsilon&=C H(u_t^\epsilon)\circ db_t^\epsilon+{1\over \sqrt \epsilon} H(u_t^\epsilon)V(x_t^\epsilon, g_t^\epsilon)dt+
 {1\over \sqrt \epsilon}A_k^*(u_t^\epsilon)\circ dw_t^k+{1\over \epsilon} A_0^*(u_t^\epsilon)dt\\
dx_t^\epsilon&=C u_t^\epsilon\circ db_t^\epsilon+{1\over \sqrt \epsilon} u_t^\epsilon V(x_t^\epsilon, g_t^\epsilon)dt.
\end{split}}$$
Below we compute a simple case. 
The argument, with suitable adjustments,  remains valid for the general case. 

\begin{example}
For simplicity consider $\R^n\times SO(n)$ with the standard  connection, and the SDE
\begin{equation}\label{J}
{\begin{split}
dg_t^\epsilon&={1\over \sqrt \epsilon}g_t^\epsilon A_k \circ dw_t^k\\
dx_t^\epsilon&=\delta g_t^\epsilon\circ db_t+{1\over \sqrt \epsilon} g_t^\epsilon V(x_t^\epsilon, g_t^\epsilon)dt.
\end{split}}
\end{equation}
Here $V$ is a $\R^n$ valued function such that $\int_G gV(x,g) dg=0$ where $dg$ is the Haar measure. For example take $V(g)$ to be a function of even powers of $g$. We assume that $V$ is suitably bounded with its partial derivatives in $x$ suitably bounded. 
The parameter  $\delta$ is  to be chosen. 

Letting $A_k^*(g)=gA_k$.
 $\L_0={1\over 2} \sum_k (A_k^*)^2$, assume that it is ${1\over 2}\Delta^L$. Taking  $\delta=\sqrt \epsilon$, formal computation by multi scale analysis shows that :
 
 {\bf Claim.} The limiting law for $x_t^\epsilon$ is govern by the   partial differential equation on $\R^n$:
\begin{equation}
\label{ou-2}
{\partial \rho \over \partial t}=-\int L_{gV(x,g)\partial_x}\L_0^{-1}({gV(x,g)}\rho) \; dg,
\end{equation}
 where the integral is with respect to the Haar measure on $SO(n)$.

  If $\delta=1$ it ought to have, in addition, a $\Delta_M$ term on the right hand  side:
   $${\partial  \rho\over \partial t}=\Delta_M f_t-\int L_{uV(x,u)}\L_0^{-1}(L_{uV(x,u)}\rho)d\nu(u),$$
 which we do not discuss  rigorously.
A drift term in the $g$ equation can also be added.    Another interesting regime to consider is $\sum_i\delta_ig_t^\epsilon \circ db_t^i$ instead of $g_t^\epsilon \circ db_t$ with
   $\delta_i$ takes values from $\{1, \sqrt \epsilon\}$. In this case, a non-Laplacian like equation would follow. In the case that $\delta_i$ are all equal and $V(x,g)$ is independent of $x$, the system can be interpret as an intertwined pair through time scaling.

Equation (\ref{ou-2}) can be deduced by the methodology below.
 Let $f: M\to \R$ be a smooth compactly supported function and $\Delta_M$ the Laplacian on $M$. Then
 $$f( x_t^\epsilon)=f(x_0)+
 \delta \int_0^t df( g_s^\epsilon  db_s)+{1\over 2} \delta^2 \int_0^t \Delta_Mf(x_s^\epsilon) ds+ 
 {1\over \sqrt \epsilon}\int_0^t df(g_s^\epsilon  V(x_s^\epsilon, g_s^\epsilon))ds.$$
 If $h$ is solution to $\L_0h (x,g)=df_x(gV(x,g))$, then
 $$ {\begin{split}
 {1\over \sqrt \epsilon}\int_0^t df(g_s^\epsilon  V(x_s^\epsilon, g_s^\epsilon))ds
 &=\sqrt{\epsilon} h(x_t^\epsilon, u_t^\epsilon)-\sqrt{\epsilon}  h(x_0, u_0)
 -\sqrt \epsilon \delta \int_0^t \partial_x h(x_s^\epsilon, g_s^\epsilon) g_s^\epsilon \; db_s\\
& - \int_0^t \partial_g h(g_s^\epsilon A_k dw_s^k) -\sqrt \epsilon \delta^2 \int_0^t \Delta_M h(x_s^\epsilon, g_s^\epsilon) ds\\
& -\int_0^t L_{g_s^\epsilon V_s^\epsilon\partial _x}h(x_s^\epsilon, g_s^\epsilon) ds.
\end{split}}$$
Since $\delta=\sqrt \epsilon$, it is now easy to observe that  $\{x_t^\epsilon\}$ is a tight family. Since $g_t^\epsilon$
 is a fast ergodic motion and $x_t^\epsilon$ does not move much as $t\to 0$, under suitable conditions, 
 $$ \lim_{\epsilon\to 0} \lim_{t\to 0} {\E f(x_t^\epsilon)-f(x_0)\over t} 
 =\lim_{t\to 0} \lim_{\epsilon\to 0}  {\E f(x_t^\epsilon)-f(x_0)\over t}=\L_{gV\partial_x}h(x_0).$$
has the required limit.
\end{example} 

 \subsection{Perturbation to Horizontal Diffusions}
 \label{section-holonomy}
 Let $M$ be a compact connected $n$-dimensional smooth Riemannian manifold
 with connection $\nabla$. 
  Let $\varpi$ be the corresponding connection 1-form on the orthonormal frame bundle $OM$ with Lie group $G$, where $G$ is taken to be $O(n)$ or $SO(n)$ depending whether $M$ is oriented.  The horizontal bundle is integrable when and only when the curvature tensor of $\nabla$ vanishes. The Lie brackets of two fundamental horizontal vector fields will in general contribute
  to a vertical motion. However perturbation to horizontal flows can still be discussed and in this case we should consider not its projection to the manifold $M$ unless the connection $\nabla$ is flat, but its motion transversal to the holonomy  bundle.

   Let $u_0\in OM$ and $\tau: [0,1]\to M$ be a $C^1$  curve  with $\tau(0)=\pi(u_0)$. Let $\tilde \tau$ be the horizontal lift of $\tau$ through $u_0$.   The parallel displacement $\tilde \tau_1$  of $u_0$ can be written as $u_0a$ some $a\in G$. The set   of such $a$ that represents parallel displacements of $u_0$ forms a subgroup of $G$ and is called  the holonomy group with reference point $u_0\in OM$ which we denote by $\Phi(u_0)$. In another word $a\in \Phi(u_0)$
if $u_0$ and $u_0a$ are connected by a horizontal curve.  Denote by $\Phi_0(u_0)$ the restricted holonomy group which contains only  loops that are homotopic
  to the identity loop. By Theorem 4.2, in Kobayashi-Nomizu \cite{Kobayashi-NomizuI}, $\Phi(u_0)$ is a closed subgroup and a sub-manifold of $G$ with $\Phi_0(u_0)$ its identity component.  Since $M$ is connected all holonomy groups are isomorphic.

Two points of $OM$ are equivalent if they are connected by a $C^1$ horizontal curve.  For each $u$ in $OM$ let $P(u_0)$ be the holonomy bundle through $u_0$, it consists of all $u\in OM$ such that $u\sim u_0$, i.e. $u$ and $u_0$ are connected by a horizontal curve. We may consider $OM$ as disjoint union of sets of the form $P(u)$. 
Let $H=\Phi(u_0)$, which acts on $P$ on the right,  and $P/H$ be the modulus space of $P$ with respect to the equivalent relation. Then $P/H$ is a smooth manifold and it can be identified with the associated bundle with fibre $G/H$ and the equivalent relation: $(uh^{-1}, h\xi)$ where $\xi$ denotes the coset  corresponding to the identity.
We identify  $(u, a\xi)$ with the orbit in $P/H$ that contains $ua$.
Denote by $\Pi_1:P\to P/H$ the natural projection. The main task in the proof of the  theorem below is to make sense of freezing the conserved `variable' 
and  averaging out the `fast `variable'.

\begin{theorem}
 \label{Model-4}
 Let $M$ be a connected and compact Riemannian manifold with a Riemannian connection $\nabla$. Consider \begin{equation}\label{canonical-sde-2}
{ \begin{split}
du_t^\epsilon&=H(u_t^\epsilon)\circ d  b_t+H_0(u_t^\epsilon)dt+\sqrt{\epsilon} \sum_{k=1}^m Z_{k}(u_t^\epsilon)\circ d w_t^{k}+\epsilon  Z_0(u_t^\epsilon)dt,\\
u_0^\epsilon&=u_0,
\end{split}}\end{equation} 
where $H_0$ is a horizontal vector field and $Z_k$ are vertical vector fields. 
Then  $ \Pi_1(u^\epsilon_{t\over \epsilon})$  converges in law, which is identified in (\ref{holonomy-limit}) below.
Furthermore define $g_t^\epsilon$ by $u_t^\epsilon=\tilde x_t^\epsilon g_t^\epsilon$, where  $x_t^\epsilon= \pi(u_t^\epsilon)$ and $\tilde x_t^\epsilon$ its horizontal lift. The projection of
 $g_{t\over \epsilon}^\epsilon$ to the space of cosets, $G/\Phi(u_0)$, converges weakly. 
 \end{theorem}
\begin{proof}
By the holonomy theorem of Ambrose-Singer \cite{Ambrose-Singer}  the Lie algebra of $\Phi(u_0)$ is a subspace of $\g$ and is spanned by matrices of the form $\Omega_v(w_1,w_2)$ where $w_1,w_2$ are horizontal vectors at $T_P$ and $v\in P(u_0)$.  If  $u\sim v$ then $\Phi(u)=\Phi(v)$.
Let $H=\Phi(u_0)$, a manifold whose dimension is denoted $n_0$. We define a distribution $S$  on $OM$: $S=\{T_u(P(u)): u\in OM\}$. It is of constant rank,  $n+n_0$.
This distribution is differentiable and involutive and $P(u)$ is the maximal integral manifold of $S$ through $u$. Note that the holonomy bundles are translations of each other: $P(u_0a) =P(u_0)a$, $a\in G$.  If  $u$ is equivalent to $v$,  the maximal integral manifolds through them are identical. 

Let $u_t$ be the solution starting from $u_0$ of the equation
$$du_t=H(u_t)\circ d  b_t+H_0(u_t)dt.$$
Then $u_t$ is constant in $t$. To see this let $f$ be a $BC^\infty$ function on $P/\Phi(u_0)$ and denote by 
$\Pi_1:P\to P/\Phi(u_0)$ the projection.
Then $$f([u_t])=f([u_0])+\int_0^t df \left (T\Pi_1(H(u_s)\right) \circ d  b_s
+\int_0^t df \left (T\Pi_1H_0(u_s)\right) ds.$$
By the Reduction Theorem, page 83 of Kobayashi-Nomizu\cite{Kobayashi-NomizuI}, 
each holonomy bundle $P(u)$ is a reduced bundle with structure group $\Phi(u)$ and the connection in $OM$ is reducible to a connection in $P(u)$. Hence
$$T_u(P(u))=HT_uOM\oplus VT_u(p(u)).$$
 In particular we have $T\Pi_1(HTOM)=0$ and
$f([u_t])=f([u_0])$.

We have shown that the solution to the horizontal SDE stays in  $P(u_0)$ for all time. The horizontal SDE, restricted to the maximal integrable manifold $P(u_0)$, satisfies the H\"ormander conditions and is ergodic
with a unique invariant measure $\mu_{P(u_0)}$.
 Fix a point $u_0\in P$ with $x_0=\pi(u_0)$. Let  $\nu$ be  the Haar measure on $H=\Phi(u_0)$. Denote by $\nu_a$ the Haar measure on $\Phi(u_0a)$. 
Note that if $v=u_0a$ some $a\in G$, let $u\in \Phi(u_0)$ and a horizontal curve $\alpha$ with  $\alpha_0=u_0, \alpha_1=u_0g, g\in \Phi(u_0)$. Then $\beta=\alpha_0a$ is horizontal with 
$\beta_0=v$ and $\beta_1=u_0ga=v a^{-1}ga$. 
Consequently  $\Phi(u_0a)=\ad(a^{-1})\Phi(u_0)$.  If $a \in H$,  $\Phi(u_0a)=\Phi(u_0)$ and $\nu_a=\nu$.

Denote by $N_{u}$ the following fibre of the holonomy bundle $P(u_0)$:  
$$N_{u}=\pi^{-1}(x)\cap P(u_0), \qquad  x=\pi(u).$$
 Locally  $N_{u_0a}=M\times\{\ad(a^{-1}\Phi(u_0)\}$ and $\mu_{P(u_0a)}=dx\times d\nu_a$ where $dx$ is the volume measure of the manifold $M$. 
 
 Let $F:OM\to \R$ be a $BC^1$ function, we  the integral
 $$\tilde F[P(u)]:=\int_{P(u)} F d\mu_{P(u)}$$
 is defined to be a number depending on a transversal of $P(u)$.
On each fibre of the holonomy bundle $P(u_0)$ we choose a reference element  $v(x)$, which determines reference elements on  holonomy bundles $P(ua)$, due to that $v(x)a$ is an element of $P(ua)$ where $u\in \pi^{-1}(x)$.  For any $u\in P(u_0)$ there is $g\in H$ such that $u=v(\pi(u))g$. 
 We define
 $$\int_{P(u_0)} F d\mu_{P(u_0)}:=\int_M \int_{P(u_0)\cap\pi^{-1}(x)} F(v(x)g) d\nu(g)dx.$$
 The resulting number is independent of the choice of $v$. To see this let $v'$ be another choice then $v'=vh$ some $h\in H$ and $u=va=v'h^{-1}a$. Since $G$ is a compact group, the Haar measure is bi-invariant,
 $${\begin{split}\int_{P(u_0)\cap\pi^{-1}(x)} F(v(x)g) d\nu_x(g)&=\int_{P(u_0)\cap\pi^{-1}(x)} F(v' (x)h^{-1}a) d\nu_x(g)\\
&=\int_{P(u_0)\cap\pi^{-1}(x)} F(v'(x) a') d\nu_x(g').\end{split}}$$
Similarly if $u=v(x)ag\in P(u_0a)$ the following integral is well defined:
 $$\int_{P(u)} F d\mu_{P(u)}:=\int_M \int_{P(u)\cap\pi^{-1}(x)} F(v(x)ag) d\nu_a(g)dx.$$

  Evaluate $f:P/\Phi(u_0)\to \R$ at $u_t^\epsilon$, denoting $\Pi_1(u)$ by $[u]$,  to see that
 $$f([u_t^\epsilon])=f([u_0^\epsilon])+\sqrt \epsilon\int_0^t df\left(T\Pi_1\left( Z_k(u_s^\epsilon)\right) \right)\circ dw_s^k +\epsilon \int_0^t df\left(T\Pi_1\left(Z_0(u_s^\epsilon)\right)\right)ds.$$
 Let $\m$ be the Lie algebra of $H$ and let 
$A_i, i=1,\dots, n_0$ be an o.n.b. of $\m$.
Let $B_j, j=n_0+1,\dots, N$ be an o.n.b. of the vertical part of the distribution $S$ at $u_0$.  Define $A_j=\varpi_{u_0}(B_{j})\in \g$. 
Consider the family of fundamental vertical vector fields  $\{A_j^*(u), j>n_0\}$, restricted to $P(u_0)$. Then $T\Pi_1(A_i^*)=0$ for $i\le n_0$ and for $j>n_0$,
$T_u\Pi_1(A_j^*)=A_j^*([u])$.

Writing  $Z_k$ in terms of the  basis $\{A_k\}$, 
$Z_k=\sum_j \sigma_k^jA_j^*$, we have
$${\begin{split} f([u_t^\epsilon])=&f([u_0^\epsilon])+  \sqrt \epsilon \sum_k\sum_{j=n_0+1}^N \int_0^t \sigma_k^j(u_s^\epsilon)\,  df\left( A_j^*([u_s^\epsilon])\right)\circ dw_s^k \\
&+\epsilon \sum_{j=n_0+1}^N\int_0^t \sigma^j_0 (u_s^\epsilon) \,df\left(A_j^*([u_s^\epsilon] )\right)ds.\end{split}}$$
The process $[u_t^\epsilon]$ is in general not Markov. It is however clear, following the standard method as used earlier, that the probability distributions $\{[u_{\cdot \over \epsilon}^\epsilon], \epsilon>0\}$ is tight and any  sequence of $[u_{t\over\epsilon}^\epsilon]$ has a convergent sub-sequence with the same limit. The limit can be identified below. Define $a_{i,j}([u])=\sum_{k\ge 1}\int \sigma_k^i\sigma_k^j d\mu_{P(u)}$, and 
$\bar Z=\sum_{j=n_0+1}^N \bar \sigma_0^j A_j^*$. For $i,j\ge n_0$ define
$$\bar \sigma_0^j([u])= \int_{P(u)} \left(\sigma_0^j +{1\over 2}\sum_{k\ge 1}  d\sigma_k^j(Z_k) \right)d\mu_{P(u)}.$$
Then
\begin{equation}\label{holonomy-limit}
\L f ([u])= \sum_{i,j=n_0+1}^N   a_{i,j} ([u])\nabla df\left( A_j^*, A_i^*\right)+ df(Z([u])).
\end{equation}
 The rest of the proof for the convergence is similar to that of Theorem \ref{level-thm}. 

For the second statement consider  the process $g_t^\epsilon$ defined by
$u_t^\epsilon=\widetilde{ \pi(u_t^\epsilon)}g_t^\epsilon$. Let  $p: G\to G/H$ be the canonical homomorphism.  For $a\in G$ denote by  $[a]$ the left coset of $H$ that contains $a$. 
 Finally note that, if $[u]$ denotes an element of of $P/H$ that contains $u$,
 $$[u_t^\epsilon]=[\widetilde{ \pi(u_t^\epsilon)}]g_t^\epsilon=[u_0]g_t^\epsilon,$$
 and the motion in $[u_t^\epsilon]$ is essentially the motion of $g_t^\epsilon$, while the latter
 is considered to be a representative of  $G/\Phi(u_0)$. Specifically,
 $${\begin{split}
dg_t^\epsilon&=\sqrt \epsilon \sum_{k=1}^m\sum_{j=1}^N \sigma_k^j(u_t^\epsilon)g_t^\epsilon A_j \circ dw_t^k+\epsilon \sum_{j=1}^N\sigma_0^j(u_t^\epsilon)g_t^\epsilon A_j,\\
d[g_t^\epsilon]&=\sqrt \epsilon \sum_{k=1}^m\sum_{j=n_0+1}^N \sigma_k^j(u_t^\epsilon)[g_t^\epsilon] A_j \circ dw_t^k+\epsilon \sum_{j=n_0+1}^N\sigma_0^j(u_t^\epsilon)[g_t^\epsilon] A_j.
\end{split}}$$
 The second statement thus follows. \end{proof}


According to Theorem 8.2 in  Kobayashi-Nomizu\cite{Kobayashi-NomizuI} if $OM$ is connected there is a connection on $OM$ such that  $P(u)=OM$. On the other extreme if the curvature vanishes the orthonormal frame bundle $OM$  foliates. This is so for a  Lie group  with the Left or right invariant connection. 
 If   $M$ is simply connected the curvature zero case corresponds to the product bundle with the trivial connection.
 See also section 6.2 in Elworthy-LeJan-Li \cite{Elworthy-LeJan-Li-book-2} for a  discussion on the equivalence of the stochastic holonomy and holonomy and
Arnaudon-Thalmaier \cite{Arnaudon-Thalmaier} for work on Yang-Mills Fields and random holonomy.


 \bigskip
 
\noindent {\bf Acknowledgement.} It is a pleasure to thank M. Hairer, Y. Maeda and S. Rosenberg for helpful discussions, J. Norris for insisting on a scaling with an unusual limit, and D. Elworthy for discussions and inspiring references. This research was supported by the EPSRC grant EP/E058124/1.
 
 \bibliographystyle{plain}
  \bibliography{OM}

 




\end{document}